\documentclass{gtpart}
\usepackage{pinlabel}
\usepackage[all]{xy}
\usepackage{amssymb, amsmath, amsthm}
\usepackage{hyperref}
\usepackage{setspace}
\usepackage{graphicx}
\usepackage{cancel}

    \title{A colored operad for string link infection}

    \author{John Burke}
    \givenname{John}
    \surname{Burke}
    \address{Department of Mathematics, Rhode Island College, Providence, RI, USA}
    \email{jburke@ric.edu}

    \author{Robin Koytcheff}
    \givenname{Robin}
    \surname{Koytcheff}
    \address{Mathematics and Statistics, University of Victoria, Victoria, BC, Canada}
    \email{rmjk@uvic.ca}



\arxivreference{1311.4217}  
\arxivpassword{xufup}   

\newtheorem{theorem}{Theorem}[section]
\newtheorem{lemma}[theorem]{Lemma}
\newtheorem{proposition}[theorem]{Proposition}
\newtheorem{corollary}[theorem]{Corollary}

\newtheorem*{T1}{Theorem 1}
\newtheorem*{T2}{Theorem 2}

\theoremstyle{definition}
\newtheorem{remark}[theorem]{Remark}
\newtheorem{definition}[theorem]{Definition}
\newtheorem{example}[theorem]{Example}

\def\beq{\begin{eqnarray*}}
\def\eeq{\end{eqnarray*}}
\def\DD{\mathbb{D}}

\def\R{\mathbb{R}}
\def\Z{\mathbb{Z}}

\def\N{\mathbb{N}}
\def\bC{\mathbb{C}}

\def\incl{\hookrightarrow}
\def\to{\rightarrow}

\def\x{\times}
\def\d{\partial}
\def\phi{\varphi}

\def\Emb{\mathrm{Emb}}

\def\Dd{\mathcal{D}}
\def\P{\mathcal{P}}

\def\K{\mathcal{K}}

\def\Diff{\mathrm{Diff}}
\def\Aut{\mathrm{Aut}}

\def\supp{\mathrm{supp}}
\def\im{\mathrm{im}\>}

\def\C{\mathcal{C}}

\def\co{\colon\thinspace}

\def\SC{\mathcal{SC}}
\def\SP{\mathcal{SP}}
\def\O{\mathcal{O}}
\def\I{\mathcal{I}}
\def\FSL{\mathrm{FSL}}
\def\EC{\mathrm{EC}}
\def\S{\mathcal{S}}
\def\PS{\mathcal{PS}}
\def\L{\mathcal{L}}
\def\LL{\widehat{\L}}

\begin{document}
\begin{abstract}
Budney constructed an operad that encodes splicing of knots and further showed that the space of (long) knots is generated over this splicing operad by the space of torus knots and hyperbolic knots.  This generalized the satellite decomposition of knots from isotopy classes to the level of the space of knots.  Infection by string links is a generalization of splicing from knots to links.  We construct a colored operad that encodes string link infection.  We prove that a certain subspace of the space of 2-component string links is generated over a suboperad of our operad by its subspace of prime links.  This generalizes a result from joint work with Blair from isotopy classes of string links to the space of string links.  Furthermore, all the relations in the monoid of 2-string links (as determined in our joint work with Blair) are captured by our infection operad.
\end{abstract}
\maketitle


\section{Introduction}
\label{Intro}
This paper concerns operations on knots and links, particularly infection by string links.  Classically, knots and links are considered as isotopy classes of embeddings of a 1-manifold into a 3-manifold, such as $\R^3$, $D^3$, or $S^3$.  Instead of considering just isotopy classes, we consider the whole \emph{space} of links, that is the space of embeddings of a certain 1-manifold into a certain 3-manifold.
We also consider spaces parametrizing the operations and organize all of these spaces via the concept of an operad (or colored operad).  The operad framework is in turn convenient for studying spaces of links and generalizing statements about isotopy classes to the space level.  Finding such statements to generalize was the motivation for recent work of the authors and R Blair on isotopy classes of string links \cite{StringLinkMonoid}.

Our work closely follows the work of Budney.
Budney first showed that the little 2-cubes operad $\C_2$ acts on the space $\K$ of (long) knots, which implies the well known commutativity of connect-sum of knots on isotopy classes.  He showed that $\K$ is freely generated over $\C_2$ by the space $\P$ of prime knots, generalizing the prime decomposition of knots of Schubert from isotopy classes to the level of the space of knots \cite{Budney}.  Later, he constructed a splicing operad $\SP$ which encodes splicing of knots.  He showed that $\K$ is freely generated over a certain suboperad of $\SP$ by the subspace of torus and hyperbolic knots, thus generalizing the satellite decomposition of knots from isotopy classes to the space level \cite{BudneySplicing}.

Infection by string links is a generalization of splicing from knots to links.  This operation is most commonly used in studying knot concordance.  One instance where string link infection arises is in the clasper surgery of Habiro \cite{Habiro}, which is related to finite-type invariants of knots and links.  In another vein, Cochran, Harvey, and Leidy observed that iterating the infection operation gives rise to a fractal-like structure \cite{CHLPrimaryDecomposition}.  This motivated our work, and we provide another perspective on the structure arising from string link infection.  We do this by constructing a colored operad which encodes this infection operation.  We then prove a statement that decomposes part of the space of 2-component string links via our colored operad.

Splicing and infection are both generalizations of the connect-sum operation.  The latter is always a well defined operation on isotopy classes of knots, but if one considers long knots, it is even well defined on the knots themselves.  This connect-sum operation (i.e., ``stacking'') is also well defined for long (a.k.a.~string) links with any number of components.  Thus we restrict our attention to string links.

\subsection{Basic definitions and remarks}
Let $I=[-1,1]$ and let $D^2 \subset \R^2 \cong \bC$ be the unit disk with boundary.
\begin{definition}
\label{StringLink}
A \emph{$c$-component string link} (or \emph{$c$-string link}) is a proper embedding of $c$ disjoint intervals
\[
\coprod_c I \incl I \x D^2
\]
whose values and derivatives of all orders at the boundary points agree with those of a fixed embedding $i_c$.  
For concreteness, we take $i_c$ to be the map which on the $i^\mathrm{th}$ copy of $I$ is given by $t \mapsto (t, x_i)$ 
where $x_i = \left(\frac{i-1}{c}, 0\right)$.  We will call $i_c$ the trivial string link.
Another example of a string link is shown in Figure \ref{stringlink}.
\end{definition}

\begin{figure}[h!]
\begin{picture}(300,108)
\put(37,0){\includegraphics[scale=.7]{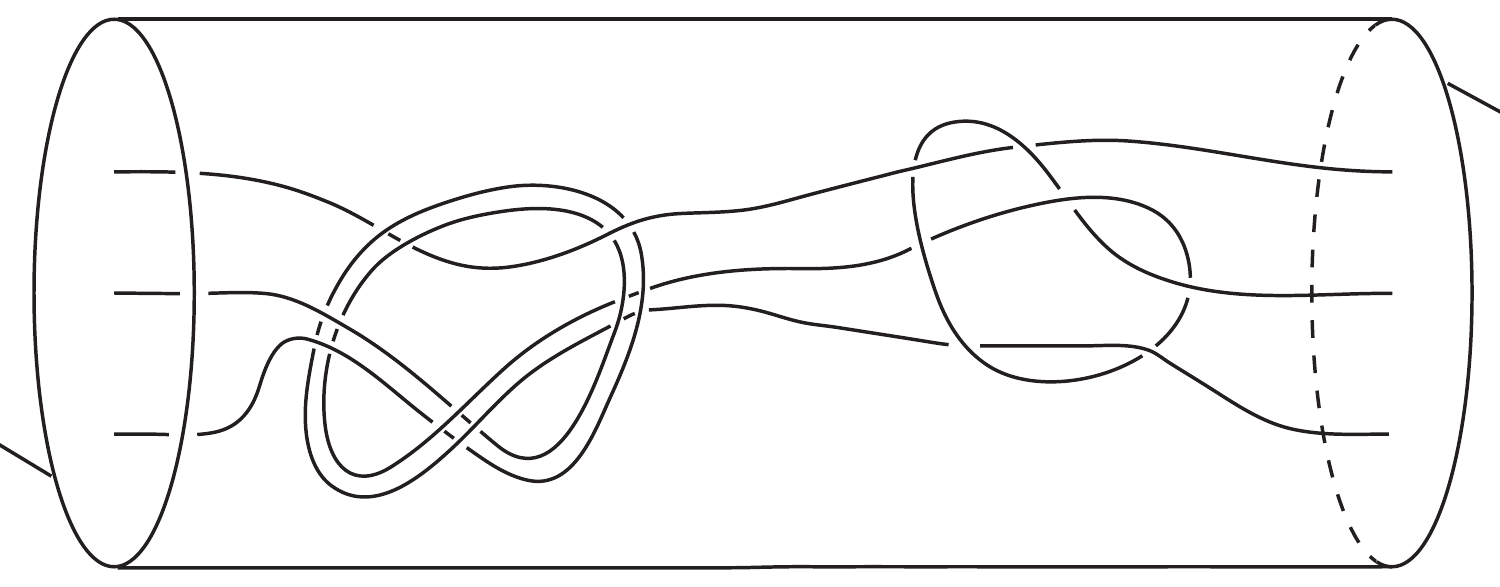}}
\put(-3,33){$D^2 \x \{0\}$}
\put(335,85){$D^2 \x \{1\}$}
\end{picture}
\caption{A string link}
\label{stringlink}
\end{figure}
In our work \cite{StringLinkMonoid}, our definition of string links allowed more general embeddings, and the ones defined above were called ``pure string links.''  We choose the definition above in this paper because infection by string links behaves more nicely with this more restrictive notion of string link.  (Specifically, it preserves the number of components in the infected link.)

The condition on derivatives is not always required in the literature.\footnote{The homotopy type of the space of such embeddings would be unchanged by omitting the condition on derivatives, since the space of possible tangent vectors and higher-order derivatives at the boundary is contractible.}  We impose it because this allows us to identify a $c$-string link with an embedding $\coprod_c \R \incl \R \x D^2$ which agrees with a fixed embedding outside of $I \x D^2$.  Let $\L_c = \Emb(\coprod_c \R, \R\x D^2)$ denote the space of $c$-string links, equipped with the $C^\infty$ Whitney topology.  An isotopy of string links is a path in this space, so the path components
of $\L_c$ are precisely the isotopy classes of $c$-string links.  Often we will write $\K$ for the space $\L_1$ of long knots.

The braids which qualify as string links under Definition \ref{StringLink} are precisely the pure braids. 
There is a map from $\L_c$ to the space $\Emb(\coprod_c S^1, \R^3)$ of closed links in $\R^3$ by taking the closure of a string link.  When $c=1$, this map is an isomorphism on $\pi_0$.  In other words, isotopy classes of long knots correspond to isotopy classes of closed knots.
In general, this map is easily seen to be surjective on $\pi_0$, but it is not injective on $\pi_0$.  For example, any string link and its conjugation by a pure braid yield isotopic closed links, and for $c \geq 3$, there are conjugations of string links by braids which are not isotopic to the original string link.
We will sometimes write just ``link'' rather than ``string link'' or ``closed link'' when the type of link is either clear from the context or unimportant.

\subsection{Main results}
Our first main result is the construction of a colored operad encoding string link infection.  An operad  $\O$ consists of spaces $\O(n)$ of $n$-ary operations for all $n\in \N$.  Roughly, an operad acts on a space $X$ if each $\O(n)$ can parametrize ways of multiplying $n$ elements in $X$.  (We provide thorough definitions in Section \ref{Operads}.)  A colored operad arises when different types of inputs must be treated differently.  In our case, we have to treat string links with different numbers of components differently, so the colors in our colored operad are the natural numbers.  This theorem is proven as Theorem \ref{DefnThm}  and Proposition \ref{CommRelations}.

\begin{T1}
There is a colored operad $\I$ which encodes the infection operation and acts on spaces of string links $\L_c$ for $c=1,2,3,...$.
\begin{itemize}
\item
When restricting to the color 1, the (ordinary) operad $\I_{\{1\}}$ which we recover is Budney's splicing operad, and the action of $\I_{\{1\}}$ on $\K$ is the same as Budney's splicing operad action.
\item
For any $c$, the operad $\I_{\{c\}}$ obtained by restricting to $c$ is an operad which admits a map from the little intervals operad $\C_1$.  The resulting $\C_1$-action on $\L_c$ encodes the operation of stacking string links.
\item
On the level of $\pi_0$, our infection operad encodes all the relations in the whole 2-string link monoid.
\end{itemize}
\end{T1}

We then use our colored operad to decompose part of the space of string links.  We rely on an analogue of prime decomposition for 2-string links proven in our joint work with R Blair \cite{StringLinkMonoid}, so we must restrict to $c=2$.
We consider a ``stacking operad'' $\I_\#$, which is a suboperad of $\I_{\{2\}}$ and which is homeomorphic to the little intervals operad.  This operad simply encodes the operation of stacking 2-string links in $I \x D^2$, with the little intervals acting in the $I$ factor.
The theorem below is proven as Theorem \ref{DecompThm}.

\begin{T2}
Let $\pi_0 \S_2$ denote the submonoid of $\pi_0 \L_2$ generated by those prime 2-string links which are not central.  (By \cite{StringLinkMonoid}, this monoid is free.)  Let $\S_2$ be the subspace of $\L_2$ consisting of the path components of $\L_2$ that are in $\pi_0 \S_2$.
Then $\pi_0 \S_2$ is freely generated as a monoid over the stacking suboperad $\I_\#$;
The generating space is the subspace consisting of those components in $\S_2$ which correspond to prime string links.
\end{T2}

\subsection{Organization of the paper}
In Section \ref{Infection}, we review the definition of string link infection.

In Section \ref{Operads}, we review the definitions of an operad and the particular example of the little cubes operad.  We then give the more general definition of a colored operad.

In Section \ref{Budney}, we review Budney's operad actions on the space of knots.  This includes his action of the little 2-cubes operad, as well as the action of his splicing operad.

In Section \ref{InfectionOperad}, we define our colored operad for infection and prove Theorem 1.  We make some remarks about our operad related to pure braids and rational tangles, and we briefly discuss a generalization to embedding spaces of more general manifolds.

In Section \ref{Decomp}, we focus on the space of 2-string links.  We prove Theorem 2, which decomposes part of the space of 2-string links in terms of a suboperad of our infection colored operad.  We conclude with several other statements about the homotopy type of certain components of the space of 2-string links.  \\

\textbf{Notation:}
\begin{itemize}
\item
$\coprod_c X$ means $\underset{\mbox{$c$ times}}{\underbrace{X \sqcup ... \sqcup X}}$
\item
$f|A$ means the restriction of $f$ to $A$
\item
$\overline{X}$ denotes the closure of $X$; $\overset{\circ}{X}$ denotes the interior of $X$
\item
 $[a]$ denotes the equivalence class represented by an element $a$; $[a_1,..a_n]$ denotes the equivalence class of a tuple $(a_1,...,a_n)$.
\end{itemize}

\subsection{Acknowledgments}
The authors thank Tom Goodwillie for useful comments and conversations.  They thank Ryan Budney for useful explanations and especially for his work which inspired this project.  They thank Ryan Blair for useful conversations and for invaluable contributions in their joint work with him, on which Theorem 2 depends.  They thank a referee for a careful reading of the paper and useful comments.  They thank Connie Leidy for suggesting the rough idea of this project.  They thank Slava Krushkal for suggesting terminology and for pointing out the work of Habiro.  Finally, they thank David White for introducing the authors to each other.  The second author was supported partly by NSF grant DMS-1004610 and partly by a PIMS Postdoctoral Fellowship.

\section{Infection}
\label{Infection}

Infection is an operation which takes a link with additional decoration together with a string link and produces a link. This operation is a generalization of splicing which in turn is a generalization of the connect-sum operation. Infection has been called multi-infection by Cochran, Friedl, and Teichner \cite{CochranFriedlTeichner}, infection by a string link by Cochran \cite{Cochran2004} and tangle sum by Cochran and Orr \cite{CochranOrr1994}. Special cases of this construction have been used extensively since the late 1970's, for example in the work of Gilmer \cite{Gilmer1983}; Livingston \cite{Livingston2005}; Cochran, Orr, and Teichner \cite{CochranOrrTeichner2003, CochranOrrTeichner2004}; Harvey \cite{Harvey2008}; and Cimasoni \cite{Cimasoni2006}. The operad we define in this paper will encode a slightly more general operation than the infection operation that has been defined in previous literature. This section is meant to
inform the reader of the definition in previous literature and
provide motivation for the infection operad.

\subsection{Splicing}

Consider a link $R \in S^3$ and a closed curve $\eta \in S^3 \setminus R$ such that $\eta$ bounds an embedded disk in $S^3$ ($\eta$ is unknoted in $S^3$) which intersects the link components transversely. Given a knot $K$, one can create a new link $R_{\eta}(K)$, with the same number of components as $R$, called the result of splicing $R$ by $K$ at $\eta$. Informally, the splicing process is defined by taking the disk in $S^3$ bounded by $\eta$; cutting $R$ along the disk; grabbing the cut strands; tying them into the knot $K$ (with no twisting among the strands) and regluing. The result of splicing given a particular $R$, $\eta$ and $K$ is show in Figure \ref{splicing}. Note that if $\eta$ simply linked one strand of $R$ then the result of the splicing would be isotopic to the connect-sum of $R$ and $K$.

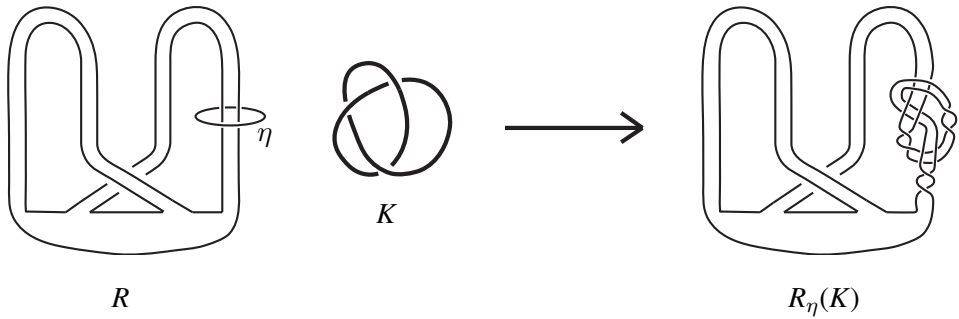
\begin{figure}[h!]
\begin{picture}(420,120)
\put(0,20){\includegraphics[scale=.8]{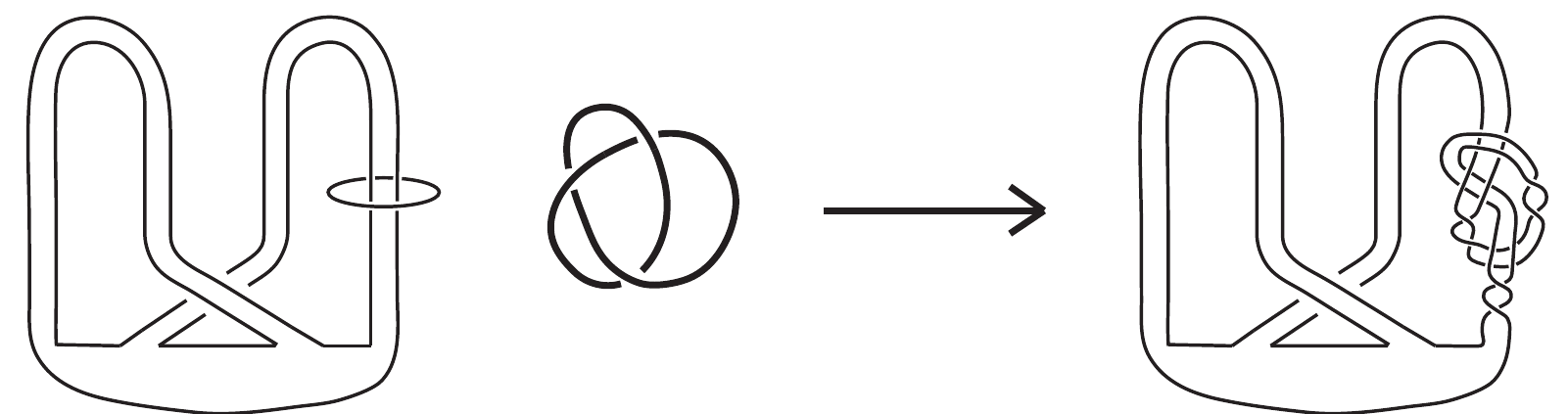}}
\put(45,0){$R$}
\put(100,63){$\eta$}
\put(145,32){$K$}
\put(301,0){$R_{\eta}(K)$}
\end{picture}
\caption{An example of the splicing operation.}
\label{splicing}
\end{figure}

Formally, $R_{\eta}(K)$ is arrived at by first removing a tubular neighborhood, $N(\eta)$, of $\eta$ from
$S^3$. Note $S^3 \setminus N(\eta) \subset S^3$ is a solid torus with $R$ embedded in its interior. Let $C_K$ denote the complement in $S^3$ of a tubular neighborhood of $K$.
Since the boundary of $C_K$ is also a torus, one can identify these two manifolds along their boundary.
In order to specify the identification, we use the terminology of meridians and longitudes.
Recall that the meridian of a knot is the simple closed curve, up to ambient isotopy, on the boundary of the complement of the knot which bounds a disk in the closure of the tubular neighborhood of the knot and has +1 linking number with the knot. Also recall that the longitude of a knot is the simple closed curve, up to ambient isotopy, on the boundary of the complement of the knot which has +1 intersection number with the meridian of the knot and has zero linking number with the knot.

The gluing of $S^3 \setminus N(\eta)$ to $C_K$ is done so that
the meridian of the boundary of $S^3 \setminus N(\eta)$
is identified with the meridian of $K$ in the boundary of $C_K$. Note that this process describes a Dehn surgery with surgery coefficient $\infty$ along $K \subset S^3$ where the solid torus glued in is $S^3 \setminus N(\eta)$. Thus, the resulting manifold will be a 3-sphere with a subset of disjoint embedded circles whose union is $R_{\eta}(K)$ (the image of $R$ under this identification).
Although the embedding of $R_\eta(K)$ in $S^3$ depends on the identification of the surgered 3-manifold with $S^3$, its isotopy class is independent of this choice of identification.

\subsection{String link infection}

Although there is a well studied generalization of the connect-sum operation from closed knots to closed links, there is no generalization of splicing by a closed link. There is, however, a generalization of splicing called infection by a string link, which we will now define. See the work of Cochran, Friedl, and Teichner \cite[Section 2.2]{CochranFriedlTeichner} for a thorough reference.

By an \emph{$r$-multi-disk} $\DD$ we mean the oriented disk $D^2$ together with $r$ ordered embedded open disks $D_1, \dots D_r$ (see Figure \ref{Multidisk}). Given a link $L \subset S^3$ we say that an embedding $\varphi \co \DD \rightarrow S^3$ of an $r$-multi-disk into $S^3$ is \emph{proper} if the image of the multi-disk, denoted by $\Dd$, intersects the link components transversely and only in the images of the disks $D_1, \dots D_r$ as in Figure \ref{Multidisk}. We will refer to the image of the boundary curves of $\varphi(D_1), \dots, \varphi(D_r)$ by $\eta_1, \dots, \eta_r$.

\begin{figure}[h!]
\begin{picture}(400,255)
\put(60,0){\includegraphics[scale=1]{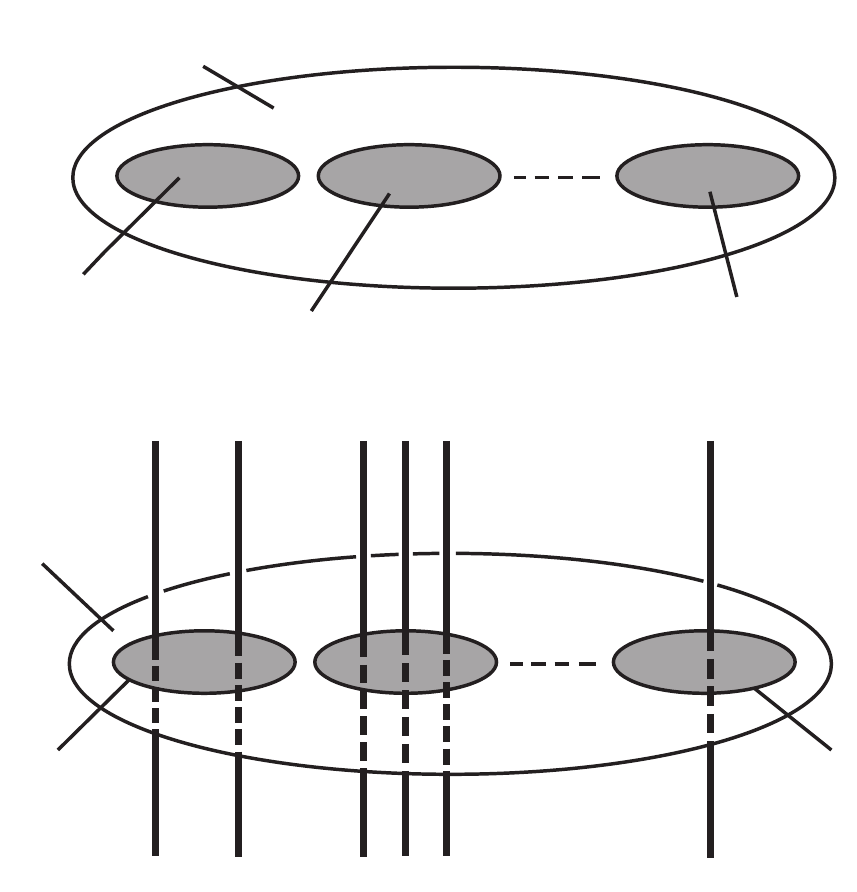}}
\put(110,240){$\DD$}
\put(73,171){$D_1$}
\put(139,160){$D_2$}
\put(270,162){$D_r$}
\put(63,93){$\Dd$}
\put(70,35){$\eta_1$}
\put(300,35){$\eta_r$}
\end{picture}
\caption{An $r$-multi-disk and a properly embedded multi-disk}
\label{Multidisk}
\end{figure}
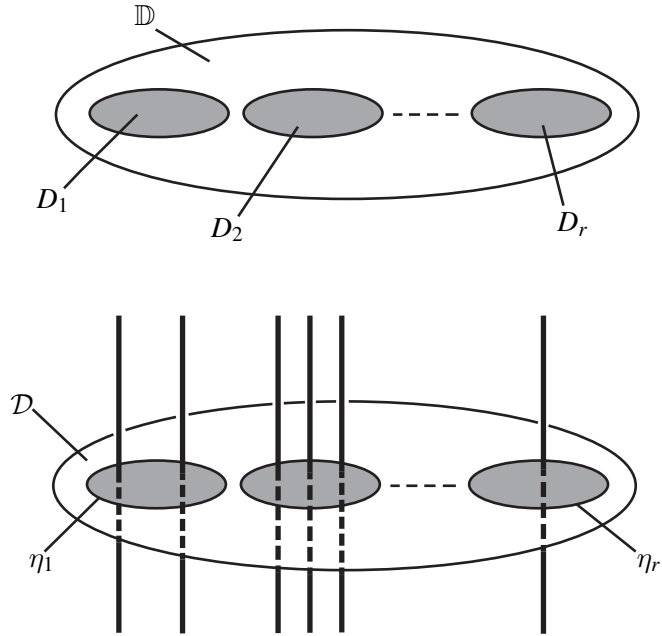

Suppose $R \subset S^3$ is link, $\Dd \subset S^3$ is the image of a properly embedded $r$-multi-disk, and $L$ is an $r$ component string link.  Then informally,
the infection of $R$ by $L$ at $\Dd$, denoted by $R_{\Dd}(L)$, is the link obtained by tying the $r$ collections of strands of $R$ that intersect the disks $\varphi(D_1), \dots, \varphi(D_r)$ into the pattern of the string link $L$, where the strands linked by $\eta_i$ are identified with the $i^\text{th}$ component of $L$, such that the $i^\text{th}$ collection of strands are parallel copies of the $i^\text{th}$ component of $L$. Figure \ref{infection} shows an example of this operation.

We now define this operation formally.
Given a string link $L \co \coprod_r I \incl I \x D^2$, let $C_L$ denote the complement of a tubular neighborhood of (the image of) $L$ in $I \x D^2$. In Figure \ref{stringlinkwithcomplement} an example of a string link is shown with its complement to the right. The meridian of a component of a string link is the simple closed curve, up to ambient isotopy, on the $\partial D^2 \times I$ boundary of the closure of the tubular neighborhood of the component which bounds a disk and has +1 linking number with the component. We call the set of such meridians the meridians of the string link. The longitude of a component of a string link is a properly embedded line segment $f\co I \rightarrow \partial D^2 \times I$, up to ambient isotopy, on the $\partial D^2 \times I$ boundary of the closure of the tubular neighborhood of the component; it is required to have +1 intersection number with the meridian of that component, to have zero linking number with that component, and to satisfy $f(0) = (1,0) \in \partial D^2 \times \{0\}$ and $f(1) = (1,1) \in \partial D^2 \times \{1\}$.  We call the set of such longitudes the longitudes of the string link. In Figure \ref{stringlinkwithcomplement} the meridians, $\mu_i$, and longitudes, $\ell_i$, are shown on the boundary of the complement. Note that the boundary of the complement of any $r$-component string link is homeomorphic to a genus-$r$ orientable surface.

\begin{figure}[h!]
\begin{picture}(300,135)
\put(43,0){\includegraphics[scale=1]{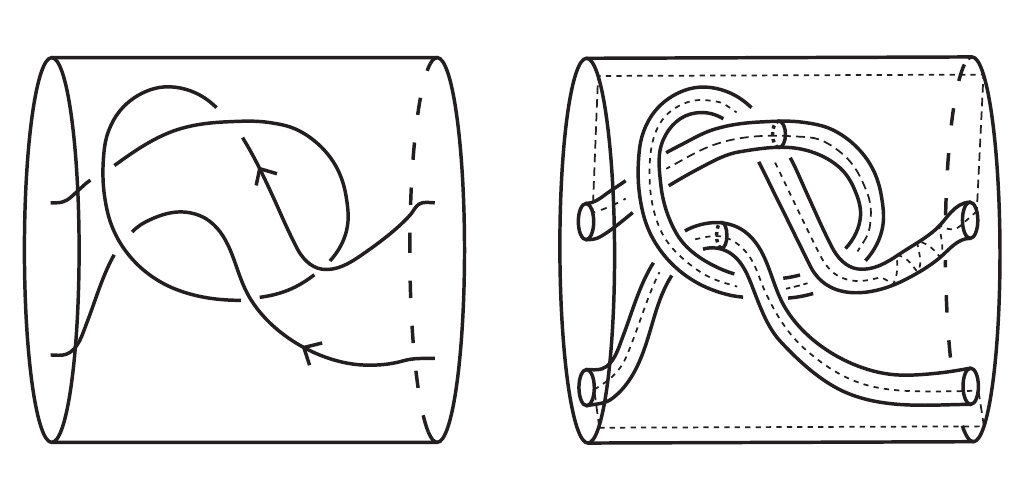}}
\end{picture}
\put(-27,25){$\footnotesize{\ell_1}$}
\put(-54,85){$\footnotesize{\mu_1}$}
\put(-35,114){$\footnotesize{\mu_2}$}
\put(3,85){$\footnotesize{\ell_2}$}
\caption{A string link and its complement.}
\label{stringlinkwithcomplement}
\end{figure}
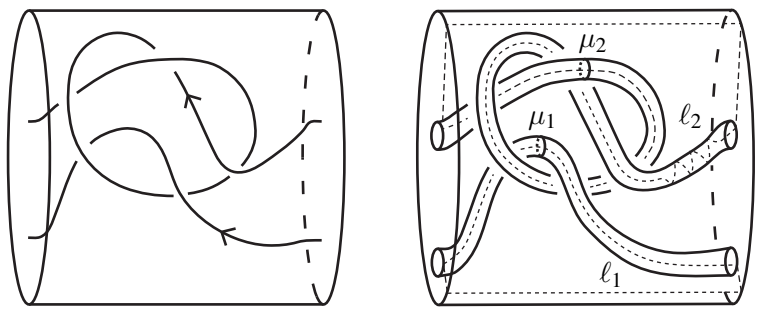

Let $R \subset S^3$ be a link, and let $L\co \coprod_r I \incl I \x D^2$ be a string link.
Fix a proper embedding of a thickened $r$-multidisk $\Dd \x I$ in $S^3 \setminus R$.
Formally the infection of $R$ by $L$ at $\Dd$ is obtained by removing $(\Dd \setminus \sqcup_i \varphi(D_1)) \times I$ from $S^3$ and gluing in the complement of $L$. Note that $(\Dd \setminus \sqcup_i \varphi(D_i))\times I$ is the complement of a $r$-component trivial string link $T$ (see Figure \ref{infection}), and thus the boundary of $S^3 \setminus ((\Dd \setminus \sqcup_i \varphi(D_1))\times I)$ is a genus-$r$ orientable surface. One identifies this boundary and the boundary of the complement of $L$, $\partial C_L$, first by identifying $\partial \Dd \times I$ with $\partial D^2 \times I$ subset of the boundary of $C_L$ where $\partial D^2 \times I$ is taken to be a subset of the boundary of $D^2 \times I$ where $L$ lives, $(\Dd \setminus \sqcup_i \varphi(D_i)) \times \{0,1\}$ is identified with $(D^2 \setminus N(L)) \times \{0,1\}$ and the $D^2 \times I$ components of the closure of $N(T)$ and $N(L)$ are identified so that the meridians and longitudes of $L$ are identified with the meridians and longitudes of $T$.

We claim that the resulting manifold is $S^3$ containing a link $R_{\Dd}(L)$ (the image of $R$ under this identification). The resulting manifold is homeomorphic to $S^3$ because
\begin{align*}
& S^3 \setminus Int((\Dd \setminus \sqcup_i \varphi(D_1)) \times I) \,\,\, \cup \,\,\, (D^2 \times I) \setminus N(L) \\
= &(S^3 \setminus \Dd \times I)  \,\,\, \cup   \,\,\,   \left( (\sqcup_i(\varphi(D_i) \times I))   \,\,\,  \cup \,\,\,    (D^2\times I)\setminus N(L)\right) \\
\cong & S^3
\end{align*}
where the last homeomorphism follows form the observation that the previous space is the union of two 3-balls.
Again, the specific embedding of $R_{\mathcal{D}}(L)$ will depend on the choice of homeomorphism, but all choices will yield isotopic embeddings.



\begin{figure}[h!]
\begin{picture}(450,300)
\put(80,0){\includegraphics[scale=.7]{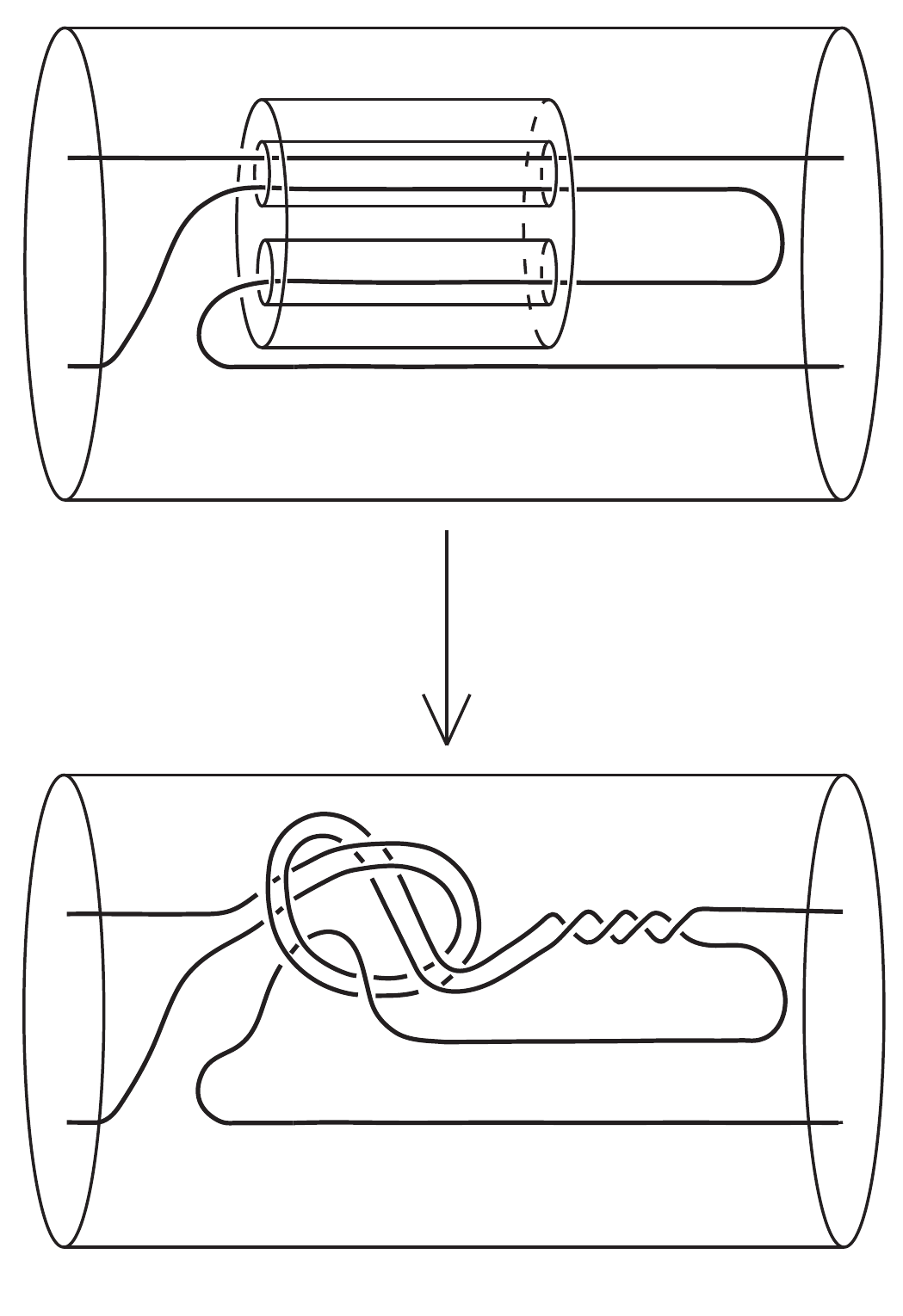}}
\put(100,183){$R$}
\put(129,284){$\Dd$}
\end{picture}
\caption{Infection of the string link $R$ along $\Dd$ by the string link $L$ from Figure \ref{stringlinkwithcomplement}.}
\label{infection}
\end{figure}

\section{Operads}
\label{Operads}
We start by reviewing the definitions of an operad $\O (=\{\O(n)\}_{n\in \N})$, and 
an action of $\O$ on $X$
(a.k.a.~an algebra $X$ over $\O$).  We then proceed to colored operads.  Technically, the definition of a colored operad subsumes the definition of an ordinary operad, but for ease of readability, we first present ordinary operads.  Readers familiar with these concepts may safely skip this Section.

\subsection{Operads}
  Operads can be defined in any symmetric monoidal category, but we will only consider the category of topological spaces.  In this case, the rough idea is as follows.  An algebra $X$ over an operad $\O$ is a space with a multiplication $X\x X\to X$, and the space $\O(n)$ parametrizes ways of multiplying $n$ elements of $X$, i.e., maps $X^n\to X$.  In other words, $\O(n)$ captures homotopies between different ways of multiplying the elements, as well as homotopies between these homotopies, etc.  Thus an element of $\O(n)$ is an operation with $n$ inputs and one output.  This can be visualized as a tree with $n$ leaves and a root, and in fact, free operads are certain spaces of decorated trees.
For a more detailed introduction, the reader may wish to consult the book of Markl, Shnider, and Stasheff \cite{MarklShniderStasheff}, May's book \cite{goils}, or the expository paper of McClure and Smith \cite{MSIntro}.

\begin{definition}
An operad $\O$ (in the category of spaces) consists of
\begin{itemize}
\item
a space $\O(n)$ for each $n=1,2,...$ with an action of the symmetric group $\Sigma_n$
\item
structure maps
\begin{equation}
\label{OperadMaps}
\O(n) \x \O(k_1) \x ... \x \O(k_n) \to \O(k_1 + ... + k_n)
\end{equation}
\end{itemize}
such that the following three conditions are satisfied:
\begin{itemize}
\item
\emph{Associativity}: the following diagram commutes:
\[
\xymatrix{
\O(n) \x \prod_{i=1}^n \O(k_i) \x \prod_{i=1}^n\prod_{j=1}^{k_i} \O(\ell_{i,j})  \ar[r] \ar[d]  & \O(n) \x \prod_{i=1}^n \O(\sum_{j=1}^{k_i} \ell_{i,j}) \ar[d] \\
\O(k_1+...+k_n) \x \prod_{i=1}^n \prod_{j=1}^{k_i} \O(\ell_{i,j}) \ar[r] & \O(\sum_{i=1}^n\sum_{j=1}^{k_i} \ell_{i,j})
}
\]
\item
\emph{Symmetry}:
Let $\sigma \x \sigma$ denote the diagonal action on the product $\O(n) \x (\O(k_1)\x...\x \O(k_n))$
coming from the actions of $\Sigma_n$ on $\O(n)$ and on $\O(k_1)\x...\x \O(k_n)$ by permuting the factors.  For a partition $\vec{k}=(k_1,...,k_n)$ of a natural number $k_1+...+k_n$, let $\sigma_{\vec{k}} \in \Sigma_{k_1+...+k_n}$ denote the ``block permutation''  induced by $\sigma$ and the partition $\vec{k}$. 

We require that the following composition agrees with the map (\ref{OperadMaps}):
\[
\xymatrix{
\O(n) \x \prod_{i=1}^n \O(k_i)  \ar[r]^{\sigma \x \sigma} & \O(n) \x \prod_{i=1}^n \O(k_{\sigma(i)}) \ar[r] & \O(\sum_{i=1}^n k_i) \ar[r]^{\sigma_{\vec{k}}^{-1}}& \O(\sum_{i=1}^n k_i)
}
\]
We also require that for $\tau_i \in \Sigma_{k_i}$ for $i=1,...,n$, the following diagram commutes:
\[
\xymatrix{
\O(n) \x \prod_{i=1}^n \O(k_i)  \ar[d]_{id \x \tau_1 \x... \x \tau_n}  \ar[r] & \O(\sum_{i=1}^n k_i) \ar[d]^{\tau_1 \x ... \x \tau_n} \\
\O(n) \x \prod_{i=1}^n \O(k_i)  \ar[r]  & \O(\sum_{i=1}^n k_i)
}
\]

\item
\emph{Identity}:
There exists an element $1 \in \O(1)$ (i.e., a map $\ast \to \O(1)$) which induces the identity on $\O(k)$ via
\begin{align*}
\O(1) \x \O(k) &\to \O(k) \\
(1, L) &\mapsto L
\end{align*}
and which induces the identity on $\O(n)$ via
\begin{align*}
\O(n) \x \O(1) \x \O(1) \x ... \x \O(1) & \to \O(n) \\
(L,1,1,...,1) &\mapsto L.
\end{align*}
\qed
\end{itemize}
\end{definition}

Some authors define the structure maps via $\circ_i$ operations, i.e., plugging in just one operation into the $i^\mathrm{th}$ input, as opposed to $n$ operations into all $n$ inputs.  These $\circ_i$ maps can be recovered from the above definition by setting $k_j=1$ for all $j\neq i$ and using the identity element in $\O(1)$.

\begin{definition}
Given an operad $\O$, an 
\emph{action of $\O$ on $X$}
(also called an \emph{algebra $X$ over $\O$}) is a space $X$ together with maps
\[
\O(n) \x X^n \to X
\]
such that the following conditions are satisfied:
\begin{itemize}
\item
\emph{Associativity}:
The following diagram commutes:
\[
\xymatrix{
\O(n) \x \O(k_1)\x ... \x \O(k_n) \x X^{k_1+...+k_n} \ar[r] \ar[d] & \O(n) \x X^n \ar[d] \\
\O(k_1+...+k_n) \x X^{k_1+...+k_n} \ar[r] & X
}
\]
\item
\emph{Symmetry}:
For each $n$, the action map is $\Sigma_n$-invariant, where $\Sigma_n$ acts on $\O(n)$ by definition, on $X^n$ by permuting the factors, and on the product diagonally.  In other words, the action map descends to a map
\[
\O(n) \x_{\Sigma_n} X^n \to X
\]
\item
\emph{Identity}:
The identity element $1\in \O(1)$ together with the map
\[
\O(1) \x X \to X
\]
induce the identity map on $X$, i.e., the map takes $(1,x)\mapsto x$.
\end{itemize}
\qed
\end{definition}

\subsection{The little cubes operad}

Our infection colored operad extends Budney's splicing operad, which in turn was an extension of Budney's action of the little 2-cubes operad on the space of long knots.  Thus the little 2-cubes operad is of interest here.

\begin{definition}
The \emph{little $j$-cubes operad} $\C_j$ is the operad with $\C_j(n)$ the space of maps
\[
(L_1,...,L_n) \co \coprod_n I^j \incl I^j
\]
which are embeddings when restricted to the interiors of the $I^j$ and which are increasing affine-linear maps in each coordinate.  The structure maps are given by composition:
\begin{align*}
\C_j(n) \x \C_j(k_1)\x ... \x \C_j(k_n) &\to \C_j(k_1+...+k_n) \\
(L_1,...,L_n), (L^1_1,...,L^1_{k_1}),...,(L^n_1,...,L^n_{k_n}) &\mapsto (L_1 \circ (L^1_1,...,L^1_{k_1}),..., L_n \circ (L^n_1,...,L^n_{k_n}))
\end{align*}
\qed
\end{definition}

Note that for all $j\geq 2$, the multiplication induced by choosing (any) element in $\C_j(2)$ is commutative \emph{up to homotopy}, which can be seen via the same picture that shows that $\pi_j X$ is abelian for $j\geq 2$.

\subsection{Colored Operads}
Now we present the precise definitions of a colored operad and an action of a colored operad on a space.  This generalization of an operad is necessary to generalize Budney's operad from splicing of knots to infection by links.

\begin{definition}
\label{ColoredOperad}
A \emph{colored operad} $\O = (\O, C)$ (in the category of spaces) consists of
\begin{itemize}
\item
a set of colors $C$
\item
a space $\O(c_1,...,c_n; c)$ for each $(n+1)$-tuple $(c_1,...,c_n,c)\in C$ together with compatible maps $\O(c_1,...,c_n; c)\to \O(c_{\sigma(1)},...,c_{\sigma(n)}; c)$ for each $\sigma \in \Sigma_n$
\item
(continuous) maps
\[
\O(c_1,...,c_n;c) \x \O(d_{1,1},...,d_{1,k_1}; c_1) \x ... \x \O(d_{n,1},...,d_{n,k_n}; c_n) \to \O(d_{1,1},...,d_{n,k_n}; c)
\]
\end{itemize}
where the maps satisfy the following three conditions:
\begin{itemize}
\item
\emph{Associativity}: The map below is the same regardless of whether one first applies the structure maps to the first two factors or the last two factors on the left-hand side:
\end{itemize}
\[
\xymatrix{
\O(c_1,...,c_n;c) \x \prod_{i=1}^n \O(d_{i,1},...,d_{i,k_i}; c_i) \x \prod_{i=1}^n \prod_{j=1}^{k_i} \O(e_{i,j,1},..., e_{i,j,\ell_{i,j}}; d_{i,j})  \ar[r] &
\O\negthinspace\left(e_{1,1,1},...,e_{n,k_n, \ell_{1,k_n}}\right)
}
\]
\begin{itemize}
\item
\emph{Symmetry}:  
The following diagram below commutes.  The vertical map is induced by $\sigma$ in both the first factor and the last $n$ factors, and $\sigma_{\vec{k}} \in \Sigma_{k_1+...+k_n}$ is the block permutation induced by $\sigma$ and the partition $(k_1,...,k_n)$.
\[
\xymatrix{
\O(c_1,...,c_n; c) \x \prod_{i=1}^n \O(d_{i,1},...,d_{i,k_i}; c_i)  \ar[r] \ar[d]_{\sigma \x \sigma} & \O(d_{1,1},...,d_{n,k_n}; c) \ar[d]^{\sigma_{\vec{k}}} \\
\O(c_{\sigma(1)},...,c_{\sigma(n)}; c) \x  \prod_{i=1}^n \O(d_{\sigma(i),1},...,d_{\sigma(i), k_{\sigma(i)}}; c_{\sigma(i)})  \ar[r] &
\O(d_{1,1},...,d_{n,k_n}; c)
}
\]
We also require that, for $\tau_i \in \Sigma_{k_i}$, $i=1,...,n$, the following diagram commutes:
\[
\xymatrix{
\O(c_1,...,c_n; c) \x \prod_{i=1}^n \O(d_{i,1},...,d_{i,k_i}; c_i)  \ar[r] \ar[d]_{id \x \tau_1 \x...\x \tau_n} & \O(d_{1,1},...,d_{n,k_n}; c) \ar[d]^{\tau_1 \x... \x \tau_n} \\
\O(c_1,...,c_n; c) \x  \prod_{i=1}^n \O(d_{i,1},...,d_{i, k_i}; c_i)  \ar[r] &
\O(d_{1,1},...,d_{n,k_n}; c)
}
\]

\item
\emph{Identity}: For every $c\in C$, there is an element $1_c\in \O(c;c)$ which together with
\[
\O(c;c) \x \O(c_1,...,c_n;c) \to \O(c_1,...,c_n;c)
\]
induces the identity map on $\O(c_1,...,c_n;c)$.  We also require that the elements $1_{c_1},...,1_{c_n}$ together with
\[
\O(c_1,...,c_n;c) \x \O(c_1;c_1) \x ...\x \O(c_n;c_n) \to \O(c_1,...,c_n;c)
\]
induce the identity map on $\O(c_1,...,c_n;c)$.
\end{itemize}
\qed
\end{definition}

The colors $c_1,..,c_n$ can be thought of as the colors of the inputs, while $c$ is the color of the output.  A colored operad with $C=\{c\}$ is just an operad, where
$\O(\underset{\mbox{$n$ times}}{\underbrace{c,...,c}};c)$
is $\O(n)$.  Sometimes, for brevity, we write ``operad'' to mean ``colored operad.''

Note that if we have a colored operad $\O$ with colors $C$ and a subset $C' \subset C$, we can restrict to another colored operad $\O_{C'}$ consisting of just the spaces $\O(c_1,...,c_n;c)$ with $c_i, c \in C'$ (and the same structure maps as $\O$).


\begin{definition}
Given a colored operad $\O=(\O, C)$, 
an \emph{action} of $\O$ on $A$
(also called an $\O$-\emph{algebra} $A$) consists of a collection of spaces $\{ A_c\}_{c\in C}$ together with maps
\begin{equation}
\label{ColoredOperadAction}
\O(c_1,...,c_n; c) \x A_{c_1} \x ... \x A_{c_n} \to A_c
\end{equation}
satisfying the following conditions:
\begin{itemize}
\item
\emph{Associativity}: The following diagram commutes:
\[
\xymatrix{
\O(c_1,...,c_n;c) \x \prod_{i=1}^n \O(d_{i,1},...,d_{i,k_i}; c_i) \x \prod_{j=1}^n A_{d_{j,k_j}}
\ar[r] \ar[d]  &
\O(c_1,...,c_n; c) \x \prod_{i=1}^n A_{c_i}
\ar[d] \\
\O(d_{1,1},...,d_{n,k_n}; c) \x \prod_{j=1}^n A_{d_{j,k_j}}
\ar[r]  & A_c}
\]
\item
\emph{Symmetry}: For each $\sigma \in \Sigma_n$, the following diagram commutes, where the vertical map is induced by the $\Sigma_n$-action and permuting the factors of $A$:
\[
\xymatrix{
\O(c_1,...,c_n; c) \x A_{c_1} \x ... \x A_{c_n} \ar[r] \ar[d] &  A_c \\
\O(c_{\sigma(1)},...,c_{\sigma(n)}; c) \x A_{c_{\sigma(1)}} \x ... \x A_{c_{\sigma(n)}} \ar[ur] & }
\]
\item
\emph{Identity}:
The map induced by $1_c\in \O(c,c)$ together with $\O(c;c) \x A_c \to A_c$ is the identity on $A_c$.
\end{itemize}
\qed
\end{definition}

If we have a subset $C' \subset C$, the restriction colored operad $\O_{C'}$ acts on the collection of spaces $\{A_c\}_{c\in C'}$.

\begin{example}
A \emph{planar algebra} as in the work of Jones \cite{JonesPlanarAlg} is an algebra over a certain colored operad.  In fact, planar diagrams form a colored operad called the \emph{planar operad} $\mathcal P$.  The colors $C$ are the even natural numbers, and $\mathcal{P}(c_1,...,c_n; c)$ is the space of diagrams with $n$ holes, $c_i$ strands incident to the $i$-th boundary circle, and $c$ strands incident to the outer boundary circle.  If $A_c$ denotes the space of tangle diagrams in $D^2$ with $c$ endpoints on $\partial D^2$, then the collection $\{ A_c \}_{c\in C}$ is an example of an algebra over $\mathcal{P}$ (a.k.a.~a planar algebra).
\end{example}

\section{A review of Budney's operad actions}
\label{Budney}
\subsection{Budney's 2-cubes action}
The operation of connect-sum of knots is always well defined on isotopy classes of knots.  If one considers \emph{long} knots, one can further define connect-sum (or stacking) of knots themselves, rather than just the isotopy classes.  That is, there is a well defined map
\[
\#\co \K \x \K \to \K
\]
where $\K=\Emb(\R, \R\x D^2)$ is the space of long knots.  If one descends to isotopy classes, this operation is commutative, i.e., $\#$ is \emph{homotopy-commutative}.  See Budney's paper \cite[p.~4, Figure 2]{Budney}  for a beautiful picture of the homotopies involved.  This picture suggests that one can parametrize the operation $\#$ by $S^1 \simeq \C_2(2)$.  Thus it suggests that the little $2$-cubes operad $\C_2$ acts on $\K$.


Budney succeeded in constructing such a 2-cubes action, but to do so,  he had to consider a space of \emph{fat} long knots
\[
\mathrm{EC}(1,D^2) := \{f\co \R^1 \x D^2 \incl \R^1 \x D^2 | \>\> \supp(f)\subset I\x D^2\}
\]
where $\supp(f)$ is defined as the closure of $\{x \in \R^1 \x D^2 | f(x) \neq x\}$.
The notation $\mathrm{EC}(1,D^2)$ stands for (self-)\emph{embeddings} of $\R^1 \x D^2$ with \emph{cubical} support.
This space is equivalent to the space of \emph{framed} long knots, but one can restrict to the subspace where the linking number of the curves $f|_{\R\x(0,0)}$ and $f|_{\R\x(0,1)}$ is zero; this subspace is then equivalent to the space of long knots.

The advantage of $\EC(1,D^2)$ is that one can compose elements.  In the 2-cubes action on this space, the first coordinate of a cube acts on the $\R$ factor in $\R \x D^2$, while the second factor dictates the order of composition of embeddings.  Precisely, the action is defined as follows.  For one little cube $L$, let $L^y$ be the embedding $I\incl I$ given by projecting to the last factor.  Let $L^x$ be the  embedding $I\incl I$ given by projecting to the first factor(s).  Let $\sigma\in \Sigma_n$ be a permutation (thought of as an ordering of $\{1,...,n\}$) such that $L^y_{\sigma(1)}(0) \leq ... \leq L^y_{\sigma(n)}(0)$.  The action
\[
\C_2(n) \x \EC(1,D^2)^n \to \EC(1,D^2)
\]
is given by
\[
(L_1,...,L_n) \cdot (f_1,...,f_n) \mapsto L^x_{\sigma(n)} \circ f_{\sigma(n)} \circ (L^x_{\sigma(n)})^{-1} \circ ... \circ L^x_{\sigma(1)} \circ f_{\sigma(1)} \circ (L^x_{\sigma(1)})^{-1}.
\]

\subsection{The splicing operad}

In the above 2-cubes action, the second coordinate is only used to order the embeddings.  Thus instead of the 2-cubes operad, one could consider an operad of ``overlapping intervals'' $\C_1'$.  An element in  $\C_1'(n)$ is $n$ intervals in the unit interval, not necessarily disjoint, but with an order dictating which interval is 	 above the other when two intervals do overlap.  Precisely, an element of $\C_1'(n)$ is an equivalence class $(L_1,...,L_n, \sigma)$ where each $L_i$ is an embedding $I\incl I$ and where $\sigma\in \Sigma_n$.  Elements  $(L_1,...,L_n, \sigma)$ and  $(L_1',...,L_n', \sigma')$ are equivalent if $L_i=L_i'$ for all $i$ and if whenever $L_i$ and $L_j$ intersect, $\sigma^{-1}(i) \leq \sigma^{-1}(j) \Leftrightarrow (\sigma')^{-1}(i) \leq (\sigma')^{-1}(j)$.  It is not hard to see what the structure maps for the operad are (and they are given in Budney's paper \cite{BudneySplicing}).  Budney then easily recasts his 2-cubes action as an action of the overlapping intervals operad $\C_1'$.

The splicing operad $\SC_1^{D^2}$ (which we abbreviate for now as $\SC$) is formally similar to the overlapping intervals operad, in that $\SC(n)$ consists of equivalence classes of elements $(L_0,L_1,...,L_n, \sigma)$ with the same equivalence relation as for $\C_1'$.  In the splicing operad, however, $L_0$ is in $\EC(1,D^2)$, $L_1,...,L_n$ are embeddings $L_i\co I\x D^2 \incl I \x D^2$, and all the $L_i$ are required to satisfy a ``continuity constraint,'' as follows.  One considers $\sigma \in \Sigma_n$ as an element of $\Sigma_{n+1}=\Aut\{0,....,n\}$ which fixes 0.  If $\sigma^{-1}(i) < \sigma^{-1}(k)$ one can think of $L_i$ as inner (or first in order of composition) with respect to $L_k$.  One wants the ``round boundary'' of $L_k$ not to touch $L_i$, but for the operad to have an identity element, one needs to allow for $L_k$ to be flush around $L_i$.  The precise requirement needed is that for $0 \leq \sigma^{-1}(i) < \sigma^{-1} (k)$
\[
\overline{\im L_i \setminus \im L_k} \cap L_k(\overset{\circ}{I} \x \d D^2)= \emptyset.
\]
Note that $\SC$ is a much ``bigger'' operad than $\C_1'$.  One can think of $L_0$ as the ``starting (thickened long) knot'' for the splicing operation and of the other $L_i$ as $n$ ``hockey pucks'' with which one grabs $L_0$ and ties up into $n$ knots.  This gives a map
\[
\SC(n) \x \K^n \to \K
\]
which will define the action of the splicing operad on $\K$.  To fully construct $\SC$ as an operad, one needs the operad structure maps, which also come from the map above.  Roughly speaking, the structure maps are as follows.  Given one splicing diagram with $n$ pucks and $n$ other splicing diagrams each with $k_i$ pucks ($i=1,...,n$), put the $i^\mathrm{th}$ splicing diagram into the $i^\mathrm{th}$ puck by composing the ``starting knots'' $L_0$ and ``taking the pucks along for the ride.''  For a precise definition and pictures, the reader may either consult \cite{BudneySplicing} or read the next Section, which closely follows Budney's construction and subsumes the splicing operad.

\section{The infection colored operad}
\label{InfectionOperad}
\begin{definition}
\label{TrivialStringLink}
Fix for each $c=1,2,3,...$ a trivial $c$-component fat string link
\[
i_c\co \coprod_c I\x D^2 \incl I\x D^2.
\]
with image denoted $S_c := \im (i_c) \subset I\x D^2$.
\end{definition}
We will be more concerned with this image of the fixed trivial fat string link rather than the embedding itself.

A convenient way of choosing an $i_c$ is to fix an embedding $\coprod_c D^2 \incl D^2$ and then take the product with the identity map on $I$.  For $c\geq 2$, we choose an embedding which takes the centers of the $c$ $D^2$'s to the points $x_1,...,x_c$ from our definition (\ref{StringLink}) of string links.  Beyond that, we remain agnostic about this fixed embedding.  For $c=1$, we choose $i_1$ to be the identity map.  This will recover Budney's splicing operad from our colored operad when all the colors are 1.


Now we define the space of $c$-component fat string links to be
\[
\FSL_c := \{ f\co \coprod_c I\x D^2 \incl I\x D^2 |\> \mbox{$f$ agrees with $i_c$ on $\d I \x D^2$}\}.
\]
These are the spaces on which the infection colored operad will act.
An element of $\FSL_3$ is displayed in Figure \ref{fatstringlink}. By our condition on the fixed trivial fat string link, we can restrict $f$ to the cores of the solid cylinders to obtain an ordinary string link $f |_{I \x \{x_1,...,x_c\}}$ as in Definition \ref{StringLink}.  \\

\begin{figure}[h!]
\begin{picture}(216,130)
\put(52,0){\includegraphics[scale=1.3]{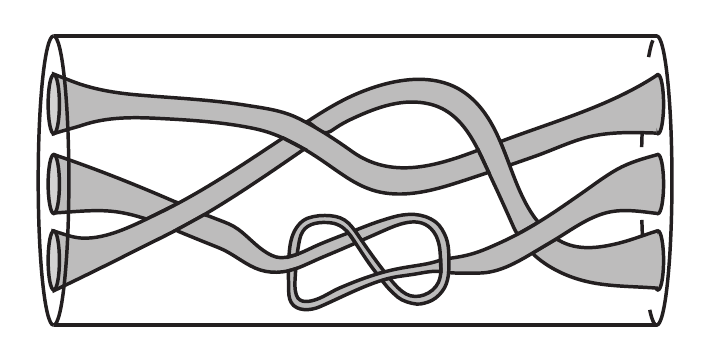}}
\end{picture}
\caption{A fat string link, or more precisely, an element of $\mathrm{FSL}_3$.}
\label{fatstringlink}
\end{figure}

\subsection{The definition of the infection colored operad}

We now define our colored operad $\I=(\I, C)$.  We put $C=\N^+$, so each color $c$ is a positive natural number.

\begin{definition}[\emph{The spaces in the colored operad $\I$}]
\label{InfectionSpaces}
An \emph{infection diagram} is a tuple $(L_0, L_1,...,L_n, \sigma)$ with $L_0 \in \FSL_c$, $\sigma\in \Sigma_n$, and $L_i$ an embedding $L_i\co I\x D^2 \incl I \x D^2$ (for $i=1,...,n$) satisfying a certain continuity constraint.  The constraint is that if $0 \leq \sigma^{-1}(i) < \sigma^{-1}(k)$, then
\begin{align*}
\overline{L_i(I \x D^2) \setminus L_k(S_{c_k})} \cap L_k(\overset{\circ}{I} \x (D^2\setminus \overset{\circ}{S_{c_k}})) = \emptyset & & (\dagger)
\end{align*}
where $S_{c_k}$ is the image of a fixed trivial string link, as in Definition \ref{TrivialStringLink}.
As in the splicing operad, we think of $\sigma\in \Sigma_n$ as a permutation in $\Sigma_{n+1}=\mathrm{Aut}\{0,1,...,n\}$ which fixes 0.

The space $\I(c_1,...,c_n;c)$ is the space of equivalence classes $[L_0,...,L_n, \sigma]$ of infection diagrams, where $(L_0,...,L_n, \sigma)$ and $(L_0',...,L_n', \sigma')$ are equivalent if $L_i=L_i'$ for all $i$, and if whenever the images of $L_i$ and $L_k$ intersect, $\sigma^{-1}(i) \leq \sigma^{-1}(k)$ if and only if $(\sigma')^{-1}(i) \leq (\sigma')^{-1}(k)$.
\qed
\end{definition}

\begin{figure}[h!]
\begin{picture}(325,190)
\put(21,0){\includegraphics[scale=.8]{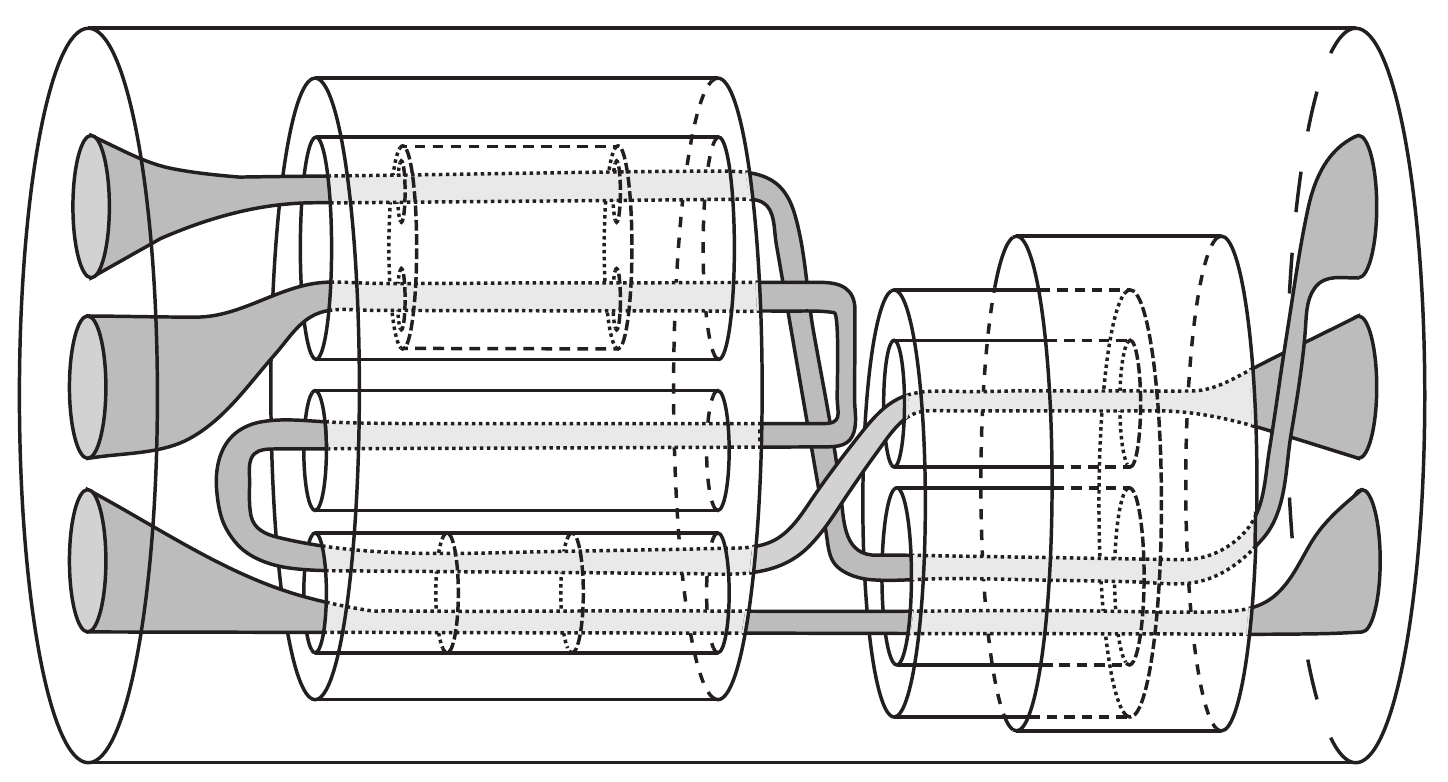}}
\put(61,147){$L_0$}
\put(101,10){$L_4$}
\put(119,21.3){$L_1$}
\put(108,93){$L_2$}
\put(218,7){$L_3$}
\put(306,7){$L_5$}
\put(219,161){$L_4(S_3)$}
\put(162,8){$L_1(S_1)$}
\end{picture}
\caption{An infection diagram, or more precisely, an element of 
$\I(1,2,2,3,1;3)$.
}
\label{infectiondiagram}
\end{figure}
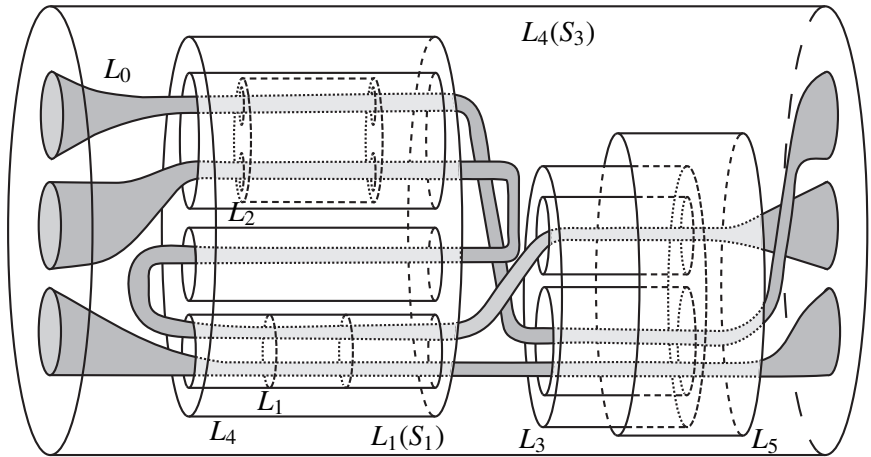

Informally, the $L_i$ are like the hockey pucks in Budney's splicing operad, and the permutation $\sigma$ is a map that sends the order of composition to the index $i$ of $L_i$.  The difference is that instead of re-embedding a hockey puck into itself, we will re-embed the image of $S_{c_i}$, a subspace of thinner inner cylinders, into the puck.  Thus we keep track of the image of $S_{c_i}$, and our pucks can be thought of as having cylindrical holes drilled in them, the holes with which we will grab the string link (or long knot) $L_0$.  Following a suggestion of V.~Krushkal, we call the restrictions of the $L_i$ to $(I\x D^2) \setminus \overset{\circ}{S_{c_i}}$ ``mufflers'' (motivated by the picture for $c_i=2$).

The generalization of Budney's continuity constraint to the constraint $(\dagger)$ is the key technical ingredient in defining our colored operad.  The need for this constraint is explained precisely in Remark \ref{ContinuityConstraintRmk} below.  The rough meaning of this condition is that a muffler which acts earlier should be inside a hole of a muffler that acts later; in other words, the ``solid part'' of a higher $L_k$ (which remains after drilling out the trivial string link) should not intersect any part of a lower $L_i$, where ``higher'' and ``lower'' are in the semi-linear ordering determined by $\sigma$.  However, we must allow for the possibility of the boundaries of the mufflers intersecting in certain ways. Figure \ref{mufflercrosssection} displays the cross-section of a set of mufflers which satisfy constraint $(\dagger)$.

\begin{figure}[h!]
\begin{picture}(325,190)
\put(80,0){\includegraphics[scale=.8]{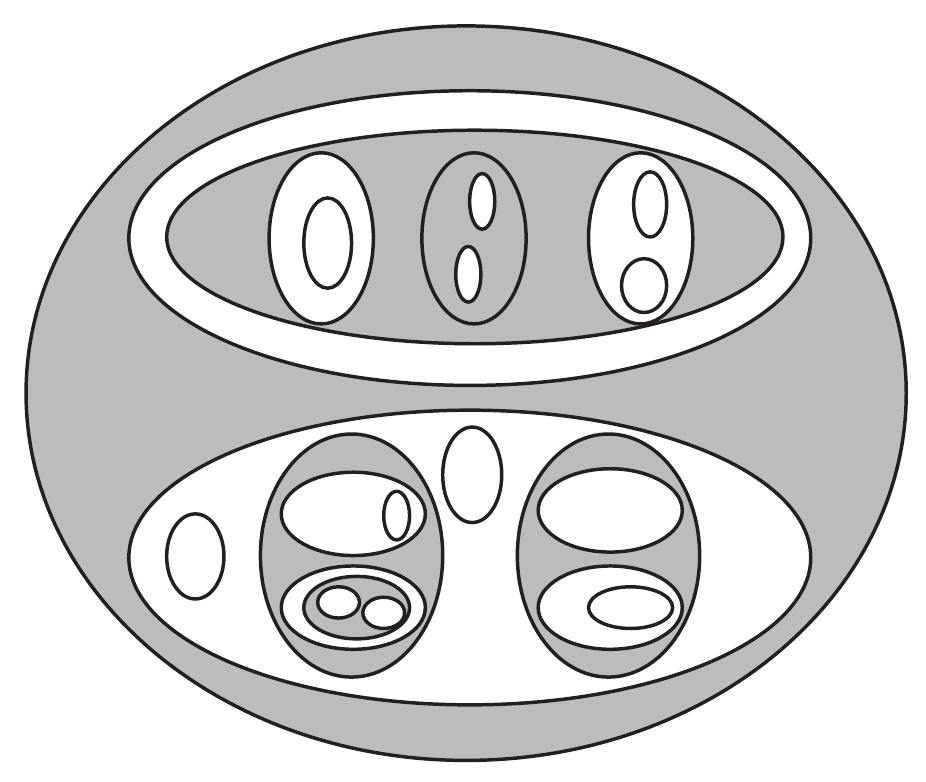}}
\end{picture}
\caption{The cross-section of a set of thirteen mufflers, including seven one-holed mufflers (or hockey pucks), satisfying the constraint $(\dagger)$.  Each grey area is the ``forbidden region'' $L_k(\overset{\circ}{I} \x (D^2\setminus \overset{\circ}{S_{c_k}}))$ of the $k^\mathrm{th}$ muffler, i.e., the region where no other muffler may lie.}
\label{mufflercrosssection}
\end{figure}

So far we haven't finished defining the operad, since we haven't defined the structure maps.  We start by defining the action on the space of fat string links.  Only after that will we define the structure maps and check that they form a colored operad and that the definition below is a colored operad action.

\begin{definition}[\emph{The action of  $\I$ on fat string links}]
\label{InfectionAction}
Consider $[L_0,L_1 ...,L_n,\sigma] \in \I(c_1,...,c_n;c)$ and fat string links $f_1,...,f_n$ where $f_k\in \FSL_{c_k}$.  Let $L_k^{in}$ be the map obtained from $L_k$ by restricting the domain to $S_{c_k}$ and restricting the codomain to its image.  We use the shorthand notation
\[
L_k \cdot f_k \quad \mbox{ to denote the map }  \quad L_k \circ f_k \circ (L_k^{in})^{-1}\co L_k(S_{c_k}) \to I \x D^2.
\]
Then we define
\begin{align*}
\I(c_1,...,c_n; c)\x \FSL_{c_1}\x ... \x \FSL_{c_n} &\to \FSL_c \\
\mbox{by} \qquad \qquad \qquad \qquad \qquad \qquad \qquad \qquad & \\
([L_0,L_1,...,L_n, \sigma], f_1,...,f_n) & \mapsto (L_{\sigma(n)}\cdot f_{\sigma(n)}) \circ ... \circ (L_{\sigma(1)}\cdot f_{\sigma(1)}) \circ L_0 .
\qed
\end{align*}
\end{definition}


\begin{figure}[p]
\begin{picture}(325,520)
\put(25,10){\includegraphics[scale=.7]{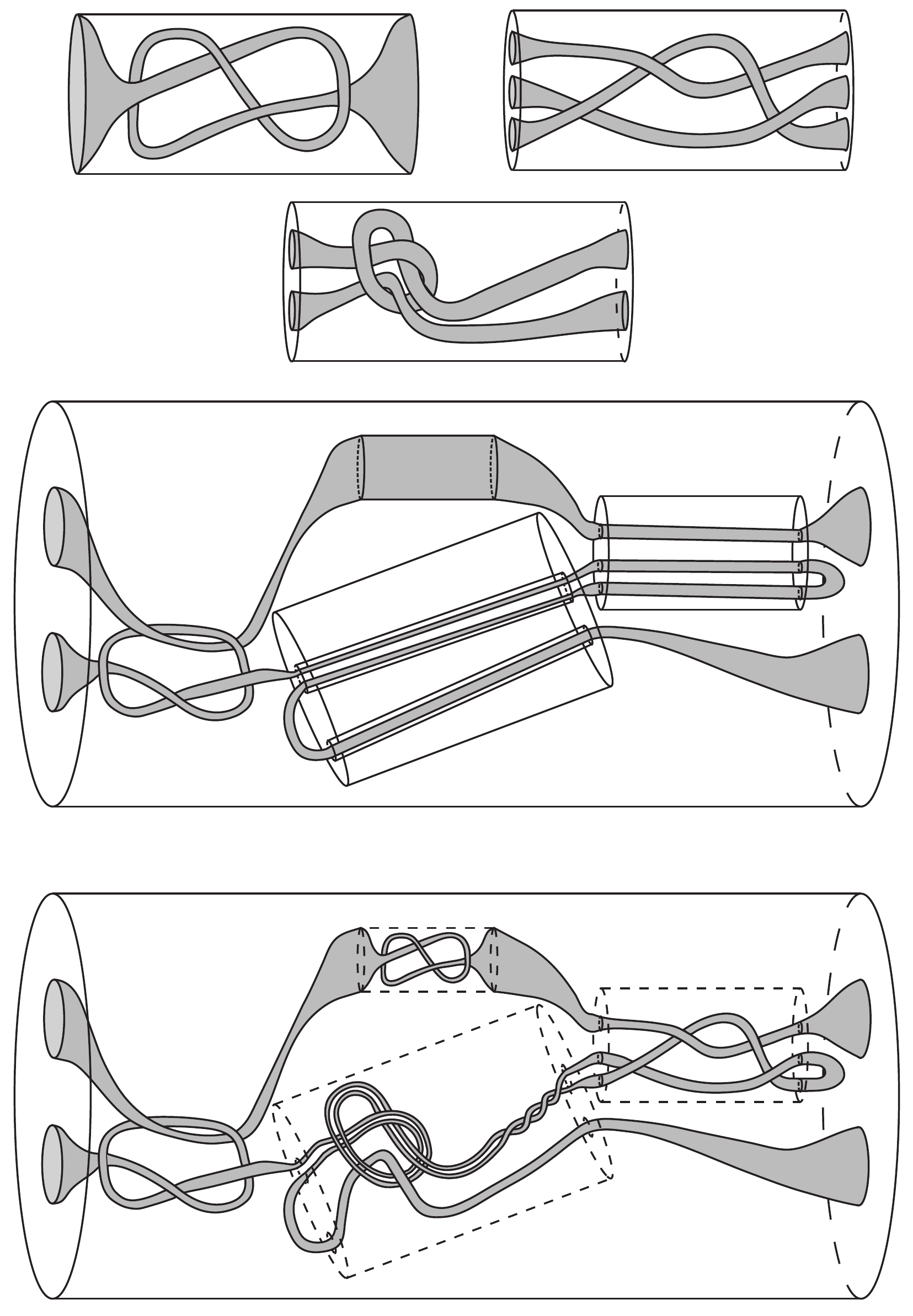}}
\put(39,455){$f_1$}
\put(341,455){$f_2$}
\put(115,382){$f_3$}
\put(55,186){$L_0$}
\put(180,336){$L_1$}
\put(203,215){$L_3$}
\put(275,317){$L_2$}
\put(135,4){$(L_3 \cdot f_3) \circ ... \circ (L_1 \cdot f_1) \circ L_0$}
\put(2015,158){$L_1 \cdot f_1$}
\put(205,35){$L_3 \cdot f_3$}
\put(275,135){$L_2 \cdot f_2$}
\end{picture}
\caption{The result of an element of an infection diagram acting on three fat string links, or more precisely a map $\I(1,3,2;2) \x \FSL_1 \x \FSL_2 \x \FSL_3 \to \FSL_2$. The 2-component fat string link at the bottom is the result of this action. }
\label{action}
\end{figure}
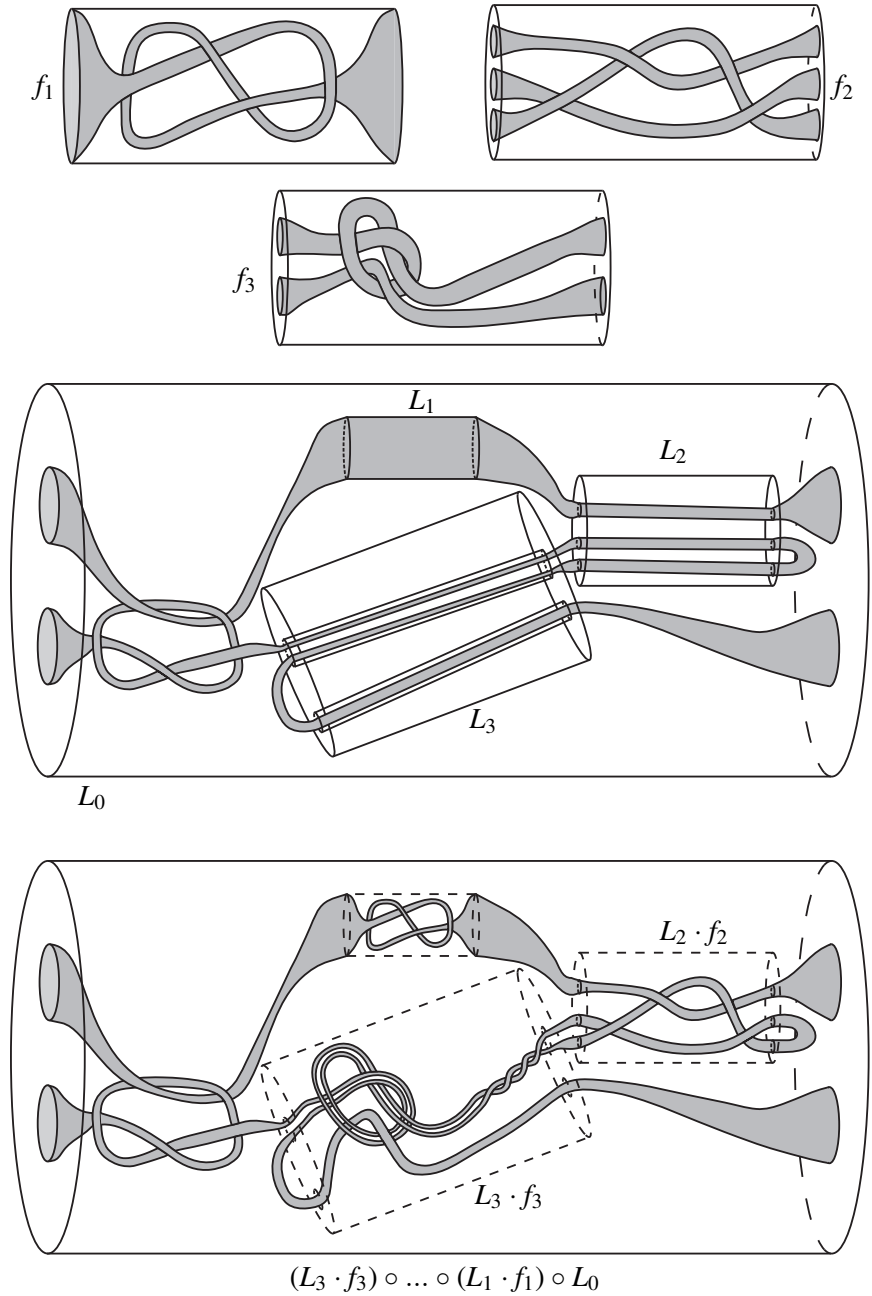

\begin{remark}
\label{ContinuityConstraintRmk}
Strictly speaking, each map $L_{\sigma(k)} \cdot f_{\sigma(k)}$ is only defined on $L_{\sigma(k)}(S_{c_{\sigma(k)}}) = \im L_{\sigma(k)}^{in}$, so one might worry whether the above composition is well defined.  We claim that the conditions on the support of the $f_{\sigma(k)}$ and the continuity constraint ($\dagger$) guarantee that we can continuously extend each $L_{\sigma(k)} \cdot f_{\sigma(k)}$ by the identity on $\im L_0 \setminus \im L_{\sigma(k)}^{in}$.

In fact, first write
\[
\d (\im L_{\sigma(k)}^{in}) = (\d I \x \coprod_{c_k} D^2) \cup (I \x \d \coprod_{c_k} D^2).
\]
Since each $f_{\sigma(k)}$ is the identity  on the $\d I \x \coprod_{c_k} D^2$ part of its domain (the ``flat boundary''), the map $L_{\sigma(k)} \cdot f_{\sigma(k)}$ is the identity on the $\d I \x \coprod_{c_k} D^2$ part of $\im L_{\sigma(k)}^{in}$.

The constraint $(\dagger)$ says that
\[
\overline{\im L_0 \setminus \im L_{\sigma(k)}^{in}}    \cap     L_{\sigma(k)}(\overset{\circ}{I} \x \d (\coprod_{c_k} D^2) ) = \emptyset,
\]
hence
\[
\overline{\im L_0 \setminus \im L_{\sigma(k)}^{in}}    \cap     \im L_{\sigma(k)} \subseteq \d I \x \coprod_{c_k} D^2.
\]
So the continuity constraint guarantees that we don't need to worry about extending past the $I \x \d \coprod D^2$ part of the boundary (the ``round boundary'').

Hence this defines the composition on the whole image of $L_0$.
\qed
\end{remark}

\begin{definition}[\emph{The structure maps in $\I$}]
\label{InfectionMaps}
The structure maps
\begin{equation}
\label{InfectionStructureMap}
\I(c_1,...,c_n;c) \x \I(d_{1,1},...,d_{1,k_1}; c_1) \x ... \x \I(d_{n,1},...,d_{n, k_n}; c_n) \to \I(d_{1,1},...,d^{n,k_n}; c)
\end{equation}
\[
(J_0,...,J_n,\rho) \x (L_{1,0},...,L_{1,k_1}, \sigma_1) \x ... \x (L_{n,0},...,L_{n,k_n}, \sigma_n) \mapsto ((J\cdot \vec{L})_0,(J\cdot \vec{L})_{1,1}, ..., (J\cdot \vec{L})_{n,k_n}, \tau)
\]
are defined as follows.  (Here $\vec{L}= (L_{1,*},...,L_{n,*})$, which can be thought of as $n$ infection diagrams, and $J\cdot \vec{L}$ is just shorthand for the result on the right-hand side.)  The ``starting'' fat string link is
\begin{align*}
(J\cdot \vec{L})_0 & := \left( \bigcirc_{i=1}^n J_{\rho(i)} \cdot L_{\rho(i),0} \right) \circ J_0 \\
&:=  (J_{\rho(n)} \circ L_{\rho(n),0} \circ (J_{\rho(n)}^{in})^{-1}) \circ ... \circ (J_{\rho(1)} \circ L_{\rho(1),0} \circ (J_{\rho(1)}^{in})^{-1}) \circ J_0.
\end{align*}
Given $a\in \{1,...,n\}$ and $b\in \{1,...,k_a\}$, the $(a,b)^{\mathrm{th}}$ puck is
\[
(J\cdot \vec{L})_{a,b} := (\bigcirc_{i=\rho^{-1}(a) +1}^n J_{\rho(i)} \cdot L_{\rho(i),0} ) \circ (J_a \circ L_{a,b})
\]
Finally, the permutation $\tau$ associated to $J\cdot \vec{L}$ is given by
\[
\tau^{-1}(a,b) :=
\tau^{-1} \left(b + \sum_{i=1}^{a-1} k_i \right) :=
\sigma^{-1}_a(b) + \sum_{i=1}^{\rho(a) -1}k_{\rho(i)}.
\]
In other words
\begin{align}
\label{PermutationList}
\tau^{-1} \co (1,1), (1,2),...,(n, k_n) \mapsto (\rho^{-1}(1), \sigma_1^{-1}(1)), (\rho^{-1}(1), \sigma_1^{-1}(2)),..., (\rho^{-1}(n), \sigma_n^{-1}(k_n))
\end{align}
where the set acted on can be thought of as a set of ordered pairs (though not a cartesian product) with a lexicographical ordering as on the left.
\qed
\end{definition}

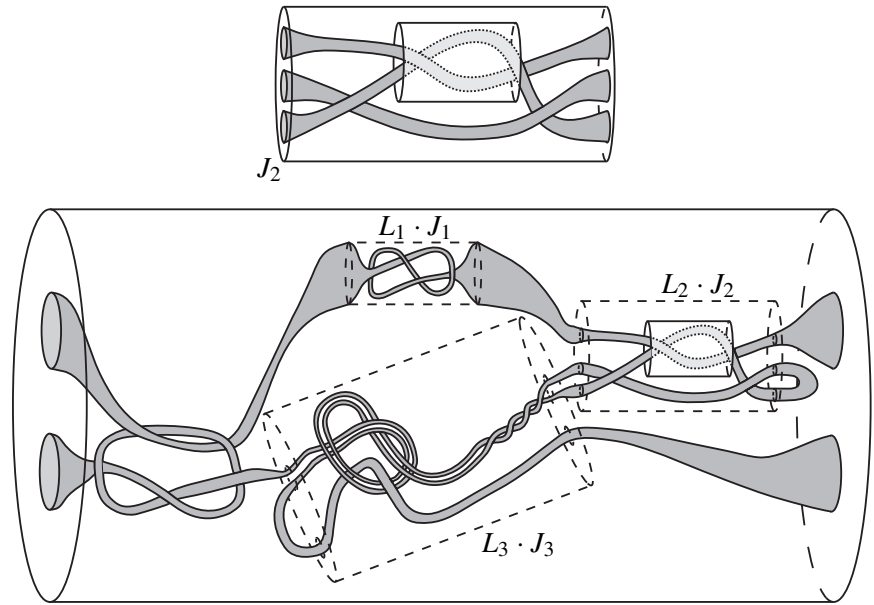
\begin{figure}[h!]
\begin{picture}(325,230)
\put(25,0){\includegraphics[scale=.7]{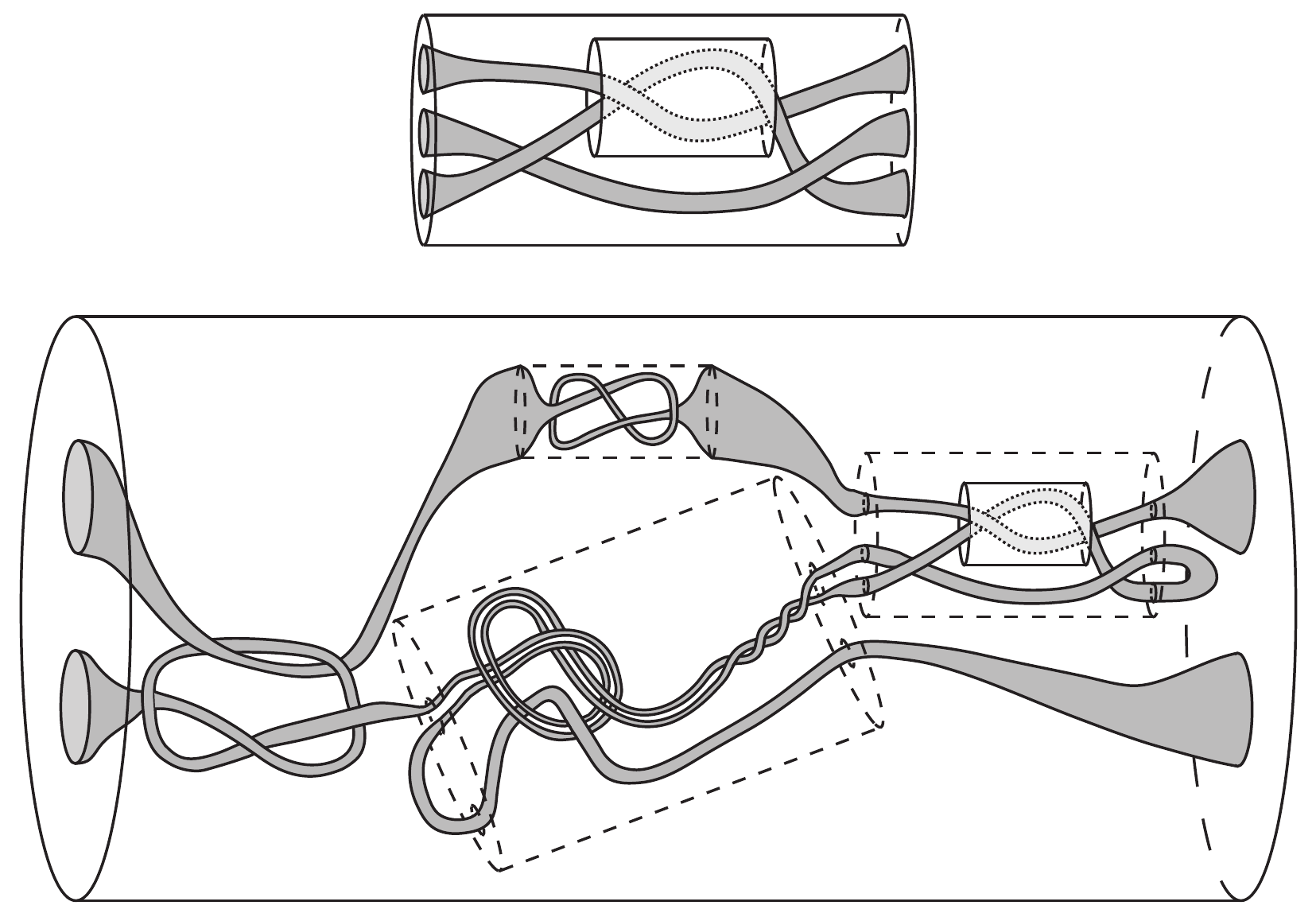}}
\put(122,165){$J_2$}
\put(168,143){$L_1 \cdot J_1$}
\put(207,23){$L_3 \cdot J_3$}
\put(275,121){$L_2 \cdot J_2$}
\end{picture}
\caption{A slight variation of Figure \ref{action}, using the same $(L_0,L_1,L_2,L_3)$ but replacing the fat string link $f_2$ in Figure \ref{action} by the infection diagram $J_2$ shown above, gives an example of the operad structure maps.
The infection diagrams $J_1$ and $J_3$ have zero mufflers, and their 0-th components are respectively $f_1$ and $f_3$.
Thus the picture above is the image of $((L_0,...,L_3),J_1,J_2,J_3)$ under the structure map
$\I(1,3,2; 2) \x \I(\emptyset; 1) \x \I(1;3) \x \I(\emptyset; 2) \to \I(1; 2)$.}
\label{structure}
\end{figure}

Notice that the action maps are just special cases of the structure maps.  In fact, $\FSL_c$ is precisely $\I(\emptyset; c)$ where $\emptyset$ is 0-tuple of positive integers (or the sequence of zero elements).  Thus each action map can be written as
\[
\I(c_1,...,c_n; c) \x \I(\emptyset; c_1) \x ... \x \I(\emptyset; c_n) \to \I(\emptyset; c)
\]
Thus we can make just a slight modification to Figure \ref{action} to produce a picture of a structure map (that is not an action map), as in Figure \ref{structure}.


\begin{theorem}
\label{DefnThm}
\begin{itemize}
\item[(A)]
The spaces and maps in Definitions \ref{InfectionSpaces} and \ref{InfectionMaps} make $\I$ a colored operad with an action on the space of fat string links given by Definition \ref{InfectionAction}.
\item[(B)]
When restricting to the single color $c=1$, one recovers Budney's splicing operad $\mathcal{SC}_1^{D^2}$.  Thus $\C_2$ maps to this part of the colored operad.
\item[(C)]
There is a map of the little intervals operad $\C_1$ to the restriction $\I_{\{c\}}$ of $\I$ to \emph{any} single color $c$.
\end{itemize}
\end{theorem}

\begin{proof}
For (A), we can first see that a composed operation (i.e., an infection diagram on the right-hand side  $\I(d_{1,1},...,d_{n,k_n}; c)$ of (\ref{InfectionStructureMap})) satisfies the constraint $(\dagger)$, as follows.  Any two non-disjoint mufflers in the composed diagram are the images of mufflers in some $\I(d_{i,1,}...,d_{i,k_i}; c_i)$ under a composition of embeddings.  But if the constraint $(\dagger)$ holds for $L_i,L_k$, then it holds for the compositions of $L_i, L_k$ with these embeddings, since ``image under an embedding'' commutes with complement, closure, and intersection.

Now we need to check the conditions of (a) associativity, (b) symmetry, and (c) identity for the structure maps.  The corresponding conditions for the action maps will then follow because the action maps are special cases of the structure maps.

(a)
Suppose we have $J=(J_0,...,J_j, \sigma)$, $\vec{L} = ((L_{1,0}, ..., L_{1,\ell_1}, \sigma_1),...,(L_{j,0},...,L_{j,\ell_j}, \sigma_j))$,
and $\vec{M} = ((M_{1,1,0},..., M_{1,1,m_{1,1}}, \tau_{1,1}), ... , (M_{j, \ell_j,0},...,  M_{j, \ell_j,m_{j,\ell_j}}, \tau_{j, \ell_j}))$.  Then $(J\cdot \vec{L}) \cdot {\vec{M}}$ has 0-th component
\[
\left( \bigcirc_{(h,k)= \nu^{-1}(1,1)}^{\nu^{-1}(j, \ell_j)}
\left( \left( \bigcirc_{i=\rho^{-1}(h)+1}^j J_{\rho(i)} \cdot L_{\rho(i),0} \right)
\circ J_h \circ L_{h,k} \right) \cdot M_{h,k,0} \right) \circ
\left( \bigcirc_{i=1}^j J_{\rho(i)} \cdot L_{\rho(i),0}\right) \circ J_0
\]
where $\nu$ is the permutation for $JL$, and where the order of the terms in the leftmost composition is given by the indices $\nu^{-1}(1,1), \nu^{-1}(1,2), ..., \nu^{-1}(j, \ell_j)$.  On the other hand, $J\cdot (\vec{L} \cdot {\vec{M}})$ has 0-th component
\[
\bigcirc_{i=1}^j J_{\rho(i)} \cdot
\left( \left( \bigcirc_{k=1}^{\ell_{\rho(i)}} L_{\rho(i),\sigma(k)} \cdot M_{\rho(i), \sigma(k),0} \right) \circ L_{\rho(i),0} \right) \circ J_0.
\]
These two expressions agree by cancelling adjacent terms $J_{\rho(i)}$, $(J_{\rho(i)}^{in})^{-1}$ and adjacent terms $L_{\rho(i),0}$, $(L_{\rho(i),0})^{-1}$ in the expression for $((J \cdot \vec{L}) \cdot \vec{M})_0$.  For example, if $J=(J_0, J_1,J_2, \iota), \vec{L} = ((L_{1,0}, L_{1,1}), (L_{2,0}, L_{2,1}),\iota), \vec{M}=((M_{1,1,0}, M_{1,1,1}), (M_{2,1,0}, M_{2,1,1}), \iota)$ (with $\iota$ denoting the identity permutation), then
\begin{align*}
((J \cdot \vec{L}) \cdot \vec{M})_0
= & [(J \cdot \vec{L})_{2,1} \cdot M_{2,1,0} ] \circ [(J \cdot \vec{L})_{1,1} \cdot M_{1,1,0} ] \circ
[J_2 \cdot L_{2,0}] \circ [J_1 \cdot L_{1,0}] \circ J_0 \\
=& [J_2 \circ L_{2,1} \circ M_{2,1,0} \circ (L_{2,1}^{in})^{-1} \circ \cancel{J_2^{-1}}] \circ \\
& [\cancel{J_2} \circ L_{2,0} \circ (J_2^{in})^{-1} \circ J_1 \circ L_1 \circ M_{1,1,0}
 \circ (L_{1,1}^{in})^{-1} \circ \cancel{J_1^{-1}} \circ \cancel{J_2} \circ \cancel{(L_{2,0})^{-1}} \circ \cancel{J_2^{-1}}] \circ \\
&[\cancel{J_2} \circ \cancel{L_{2,0}} \circ \cancel{(J_2^{in})^{-1}}] \circ [\cancel{J_1} \circ L_{1,0} \circ (J_1^{in})^{-1}] \circ J_0 \\
=& [J_2 \cdot (\vec{L} \cdot \vec{M})_{2,1,0}] \circ [J_1 \cdot (\vec{L} \cdot \vec{M})_{1,1,0}] \circ J_0 \\
= & (J \cdot (\vec{L} \cdot \vec{M}))_0
\end{align*}

Checking that the $(a,b,c)^{\mathrm{th}}$ mufflers of these two infection diagrams agree similarly involves cancelling adjacent terms in the expression for $((J \cdot \vec{L})\cdot \vec{M})_{a,b,c}$.  (Also cf. \cite{BudneySplicing}.)

Finally, to check that the permutations for these two infection diagrams agree, note that the inverse of either one is given (with notation as in (\ref{PermutationList})) by $(i,k,h) \mapsto (\rho^{-1}(i), \sigma_i^{-1}(k), \tau_{i,k}^{-1}(h))$.

(b)  We need to check that both diagrams in the symmetry condition in Definition \ref{ColoredOperad} commute for $\O=\I$.  The maps involved consist of permutations of labels on mufflers and labels on infection diagrams.  The commutativity of these diagrams is easily verified.

(c)  The identity $1_c \in \I(c;c)$ is an element $[L_0, L_1, e]$ with $L_0$ the fixed trivial $c$-component fat string link, $L_1$ the identity map on $I\x D^2$, and $e$ the element in $\Sigma_1$.

Part (B) of the theorem follows quickly from our definitions.  One can check that by choosing the identity map for the trivial fat 1-string link, our constraint $(\dagger)$ reduces to Budney's continuity constraint.  The rest of our definitions are then exactly as in Budney's splicing operad.

For part (C), the map $\C_1 \to \I_{\{c\}}$ is easy to construct.  An element of $\C_1(n)$ is $(a_1,...,a_n)$ where each $a_i\co I \incl I$ is the restriction of an affine-linear, increasing map.  The map
$\C_1 \to \I_{\{c\}}$ is given by $(a_1,...,a_n) \mapsto (i_c, a_1 \x id_{D^2},...,a_n \x id_{D^2}, \iota)$ where $i_c$ was the trivial fat $c$-string link, and where $\iota$ is the identity permutation.  (Actually, we could choose any permutation since the mufflers are disjoint.)
\end{proof}

\begin{remark}  For $c\neq 1$, it is clear that $\C_2$ cannot map to the operad $\I_c$, for then connect-sum of string links would be (homotopy-)commutative.  But this is not the case.  For $c\geq 3$, the pure braid group is not abelian, and for $c=2$, the monoid of string links up to isotopy is nonabelian.  The latter result can be deduced either from our recent results on the structure of this monoid \cite{StringLinkMonoid} or from work of Le Dimet in the late 1980's \cite{LeDimet1988} on the group of string links up to cobordism.
\qed
\end{remark}

Just as Budney's fat long knots are equivalent to framed long knots, our fat string links are equivalent to framed string links.
In more detail, given a fat string link $L \in \FSL_c$, we can restrict to the ``cores of the tubes'' to get an ordinary string link $L|(I \x \{x_1,...,x_c\})$.  Thus we have a map $\FSL_c \to \L_c$, which is a fibration, since in general restriction maps are fibrations.  The fiber $\mathrm{Fib}_L$ over $L$ is the space of tubular neighborhoods of $\im L$ which are fixed at the boundaries.  We express such a neighborhood as a map $\eta\co \coprod_c I \x D^2 \to I \x D^2$ and associate to $\eta$ a collection of $c$ loops in $SO(2)$; these are obtained by taking the derivative at $(0,0)$ of the map $\{t\} \x D^2 \to I \x D^2$, for each $t\in \coprod_c I$.  Thus we can map the fiber $\mathrm{Fib}_L$ to $(\Omega SO(2))^c$.  This ``derivative map'' is a homotopy equivalence (by shrinking $\eta$ to a small neighborhood of $\coprod_c I \x \{0\}$).  Since $\Omega SO(2) \cong \Z$, we can write the fibration as
\[
\xymatrix{ \Z^c \ar[r] & \FSL_c \ar[r] & \L_c.
}
\]
For $L \in \FSL_c$, there are $c$ framing numbers $\omega_1,...,\omega_c$, one for each component.  The $j^{\mathrm{th}}$ framing number is given by the linking number of $I_j \x (0,0)$ with $I_j \x (1,0)$, where $I_j$ is the $j^\mathrm{th}$ copy of $I$ in $\coprod_c I$.  The map $\omega_1 \x ... \x \omega_c \co \FSL_c \to \Z^c$ gives a splitting of the above fibration.  Then we consider the product fibration $\Z^c \to \L_c \x \Z^c \to \L_c$ and the map from the above fibration to this one induced by the splitting.  The long exact sequence of homotopy groups for a fibration together with the five lemma imply that the map from $\FSL_c$ to $\L_c \x \Z^c$  is a weak equivalence, hence a homotopy equivalence.  Thus $\widehat{\L}_c := (\omega_1 \x...\x \omega_c)^{-1}\{(0,0,...,0)\}$ is equivalent to $\L_c$.

\begin{corollary}
By restricting to the subspaces $\widehat{\L}_c\subset \FSL_c$ of fat string links with zero framing number in every component, we obtain an action of $\I$ on spaces homotopy-equivalent to the spaces of $c$-component string links.
\end{corollary}

\subsection{Mufflers, rational tangles, and pure braids}
\label{RationalTangles1}
We now briefly discuss how general an infection our operad $\I$ encodes.  Informally, one might wonder how twisted the inner cylinders (i.e., the holes) $L^{in}$ of a muffler could be.  Clearly, a fat string link can appear as $L^{in}$ if and only if the pair $(I\x D^2, L^{in})$ is homeomorphic to the pair $(I\x D^2, i_c)$ where $i_c$ the trivial fat $c$-string link.  The purpose of the following well known Proposition is just to show an alternative and perhaps more intuitive way of thinking about such string links.  Recall from Definition \ref{TrivialStringLink} that $S_c$ is the image of $i_c$.

\begin{proposition}
\label{RationalTFAE}
The following are equivalent:
\begin{itemize}
\item[(1)]
There is a
diffeomorphism of pairs $\xymatrix{(I \x D^2, S_c) \ar[r]^-\cong&  (I \x D^2, \im(L)).}$
\item[(2)]
There is an isotopy from $L$ to the trivial link which takes $\d (\coprod_c I)$ into
$\d (I \x D^2)$.  Note that the isotopy need not \emph{fix} the endpoints of $\coprod_c I$.
\end{itemize}
\end{proposition}

\begin{proof}

(1) $\Rightarrow$ (2):
Suppose we have a
diffeomorphism of pairs $h$ as in (1).
It suffices to show that the identity can be connected to this diffeomorphism by a path of diffeomorphisms of $I \x D^2$, for then we can restrict to $S_c$ to obtain the desired isotopy.

By Cerf's Theorem \cite{Cerf}, the space of diffeomorphisms of $S^3$ is connected.  As a corollary, so is the space of diffeomorphisms of $D^3$ whose values and derviatives agree with the identity on the boundary.  In fact, this follows by considering the fibration
\[
\mathrm{Diff}(D^3,\partial D^3) \to \mathrm{Diff}(S^3) \to \mathrm{Emb}(D^3, S^3)
\]
given by restricting to a hemisphere of $S^3$.  The base space is homotopy-equivalent to $SO(3)$, which is connected, while the fiber is the space of diffeomorphisms of $D^3$ fixed on the boundary.

Now a diffeomorphism $\phi\co (I \x D^2, S_c) \overset{\cong}{\to} (I \x D^2, \im(L))$ is clealry isotopic to one that is the identity outside of a ball $D^3$ contained in $I \x D^2$.  Combining this with Cerf's Theorem, we get a path from $\phi$ to the identity, as desired.

(2) $\Rightarrow$ (1): by the isotopy extension theorem (see for example Hirsch's text \cite{Hirsch}), an isotopy as in (3) can be extended to a diffeotopy of the whole space $I\x D^2$.  The diffeotopy at time 1 then gives the desired diffeomorphism.
\end{proof}

The 2-string links which satisfy the above condition(s) are by definition precisely the 2-string links which are also \emph{rational 2-tangles}.  (Here we consider only string links, not arbitrary tangles; the reader may consult the work of Conway \cite{Conway} for more details about rational tangles in general.)  Note that pure braids are examples of rational 2-tangles, since it is easy to see that a pure braid satisfies (2) above.
We immediately have the following result, which informally says that ``a muffler can grab the string link in the shape of 
any rational 2-tangle:''


\begin{proposition}
A fat 2-string link $L^{in}$
\[
\xymatrix{
S_c \cong \coprod_c I \x D^2 \ar@{^(->}[r]^-{L^{in}} & I \x D^2
}
\]
extends to a diffeomorphism $L$ of $I\x D^2$ if and only if the core $L^{in}|(I\x \{x_1,x_2\})$ of $L^{in}$ is a rational tangle.
\qed
\end{proposition}

\subsection{Generalizations to other embedding spaces}
 For $j \in \N^+$ and $M$ a compact manifold with boundary, let $\EC(j,M)$ be the space of ``cubical embeddings'' $\R^j \x M \incl \R^j \x M$, that is, all such embeddings which are the identity outside $I^j \x M$.  Budney constructs the actions of the little 2-cubes operad $\C_j$ and the splicing operad $\SC_1^{D^2}$ on the space of long knots as special cases of actions of the operads $\C_j$ and $\SC_j^M$ on $\EC(j,M)$.  Our extension of the splicing operad to string links also gives an extension of the more general splicing operad $\SC_j^M$ to a colored operad acting on spaces of embeddings $I^j \x \coprod_c M \incl I^j \x M$.

For each $c\in \N^+$ fix an embedding
\[
i_c\co \coprod_c I^j \x M \incl \coprod_c I^j \x M
\]
by fixing an embedding $\coprod_c M \incl M$.  Let $S_c$ be the image of $i_c$.

Let
\[
\EC^{\coprod_c}(j,M) := \{ f\co \coprod_c I^j \x M \incl I^j \x M |\> \mbox{$f$ agrees with $i_c$ on $\d I \x M$}\}.
\]

\begin{definition}[\emph{The spaces in the colored operad $\I_j^M$}]
An element in $\I_j^M(c_1,...,c_n;c)$ is an equivalence class of tuples $(L_0, L_1,...,L_n, \sigma)$.  Here $L_0 \in \EC^{\coprod_c}(j,M)$, $\sigma\in \Sigma_n$, and for $i=1,...,n$, $L_i$ is an embedding $L_i\co I^j \x M \incl I^j \x M$ subject to the constraint that for $0 \leq \sigma^{-1}(i) < \sigma^{-1}(k)$
\begin{align*}
\overline{\im L_i \setminus L_k(S_c)} \cap L_k(\overset{\circ}{I^j} \x (M\setminus \overset{\circ}{S_c})) = \emptyset & & 
\end{align*}
Here we think of $\sigma\in \Sigma_n$ as a permutation in $\Sigma_{n+1}=\mathrm{Aut}\{0,1,...,n\}$ which fixes 0.

Tuples $(L_0,...,L_n, \sigma)$ and $(L_0',...,L_n', \sigma')$ are equivalent if $L_i=L_i'$ for all $i$ and if whenever the images of $L_i$ and $L_k$ intersect, $\sigma^{-1}(i) \leq \sigma^{-1}(k)$ if and only if $(\sigma')^{-1}(i) \leq (\sigma')^{-1}(k)$.
\qed
\end{definition}

The structure maps of $\I_j^M$, as well as an action of $\I_j^M$ on the spaces $\{ \EC^{\coprod_c}(j,M)\}_{c\in \N^+}$, can be defined exactly as in the special case where $j=1$ and $M=D^2$.

\section{Decomposing the space of 2-string links using the infection operad}
\label{Decomp}

\subsection{The monoid of 2-string links}
Note that given any monoid $\mathcal{M}$ and subset $\C$ of central elements, the quotient monoid $\mathcal{M}/\C$ is well defined.  We are interested in the monoid $\mathcal{M}=\pi_0 \L_c$ of isotopy classes of $c$-string links, especially for $c=2$.  The units in $\pi_0\L_c$ are precisely the pure braids \cite[Proposition 2.7]{StringLinkMonoid}.
We say that a non-unit $c$-string link $L$ is \emph{prime} if $L=L_1\# L_2$ implies that either $L_1$ or $L_2$ is a unit (pure braid).

\begin{definition}
\label{SplitAndCable}
\begin{itemize}
\item[(1)]
A string link $L$ is \emph{split} if there exists a properly embedded 2-disk $(D, \d D) \incl (I \x D^2, \d(I \x D^2))$ whose image is disjoint from $L$ and such that the two 3-balls into which $D$ separates $I \x D^2$ each contain component(s) of $L$.  Such a 2-disk is called a \emph{splitting disk}.  See Figure \ref{SplitLink}.
\item[(2)]
A \emph{1-strand cable} is a string link $L$ which has a neighborhood $N \cong I \x D^2$ such that $L$  considered as a link in $N$ is a (pure) braid $B$.  In other words, ``all the strands are tied into a  knot.''  We call $\d N \setminus \d(I\x D^2)$ a \emph{cabling annulus} for $L$.
See Figure \ref{Cable}.
\end{itemize}
\end{definition}

\begin{figure}[h!]
\begin{picture}(216,130)
\put(112,0){\includegraphics[scale=.6]{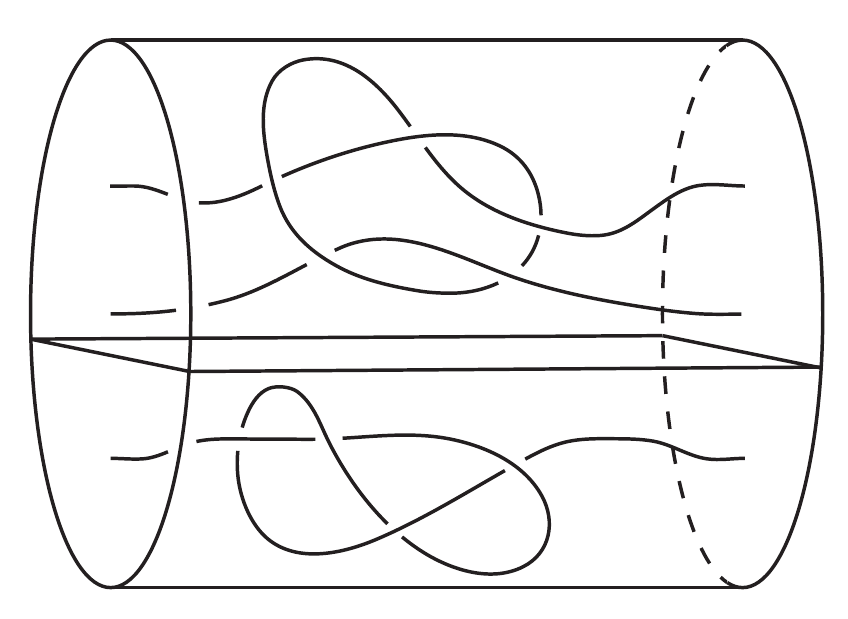}}
\end{picture}
\caption{An example of a split link.}
\label{SplitLink}
\end{figure}

\begin{figure}[h]
\begin{picture}(216,130)
\put(107,0){\includegraphics[scale=.7]{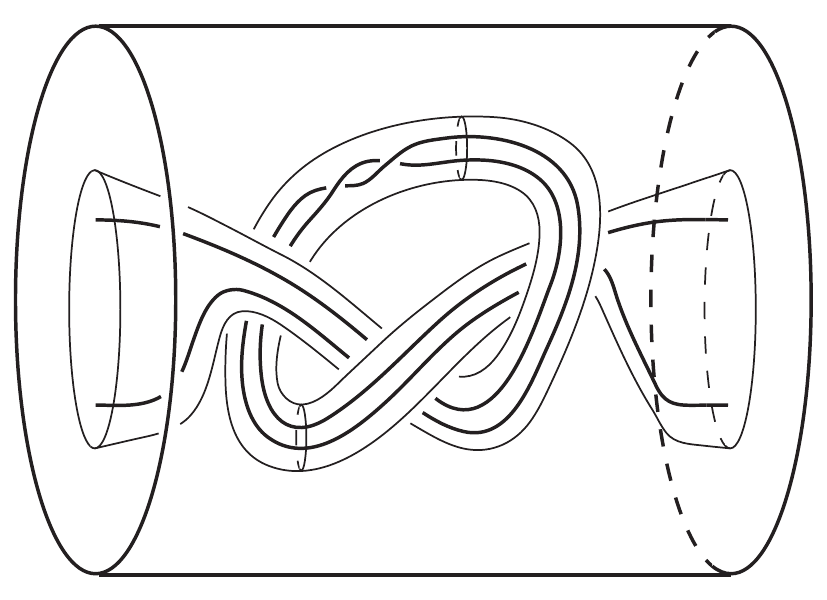}}
\end{picture}
\caption{An example of a 1-strand cable, shown together with a cabling annulus.}
\label{Cable}
\end{figure}


Since we are now focusing on 2-string links, we need not consider (or even define) $k$-strand cables for $k>1$.  Hence we will often refer to 1-strand cables as just \emph{cables}.

\begin{theorem}[proven in \cite{StringLinkMonoid}]
\label{hypo}
The monoid $\pi_0\L_2$ has center $\C$ generated by the pure braids, split links, and 1-strand cables.  The quotient $\pi_0\L_2 / \C$ is free.  Furthermore, every 2-component string link can be written as a product of prime factors
\[
L_1\#...\#L_m \# K_1 \#...\# K_{n-m}
\]
where the $K_i$ are precisely the factors which are in the center.  Such an expression is unique up to reordering the $K_i$ and multiplying any of the factors by units (pure braids).
\end{theorem}

\subsection{Removing twists}
\label{RemovingTwists}
Next note that the linking number gives a map $\ell\co \L_2\to \Z$ which descends to a monoid homomorphism $\pi_0\ell\co \pi_0 \L_2 \to \Z$. For $n\in \Z$, let $\L^n_2 = \ell^{-1}\{n\}$.  We might like to think of this as $0 \to \L^0_2 \to \L_2 \to \Z \to 0$, though if we wanted this to be a short exact sequence of monoids, we should instead write
\[
0 \to \pi_0 \L^0_2 \to \pi_0 \L_2 \to \Z \to 0,
\]
since $\L_2$ is only a monoid up to homotopy.  There is an action of $\Z$ on $\L_2$ where the generator $1\in \Z$ acts by following the embedding in $\L_2$ by the map $D^2 \x I \to D^2 \x I$ given by $(z,t) \mapsto (e^{2\pi i t} z, t)$.  The action of any $m\in \Z$ thus gives a continuous map\footnote{We can see that the map does indeed have this codomain because the resulting twisted link can be taken by an isotopy to a link where the twists are on one end, in which case the linking number is clearly increased by  $m$.}  $\L^n_2 \to \L^{n+m}_2$
with continuous inverse given by the action of $-m$.
Thus $\L_2 \cong \L^0_2 \x \Z$, and it suffices to study $\L^0_2$ to understand $\L_2$.  Note that Theorem \ref{hypo} above implies that an element of $\pi_0 \L^0_2$ can be written as a product of primes $L_1\#...\#L_m \# K_1 \#...\# K_{n-m}$ which is unique up to \emph{only reordering the $K_i$}.  We similarly define a subspace $\LL^0_2 \subset \LL_2$ in the space of fat string links with zero framing number; note that $\LL^0_2 \simeq \L^0_2$.

\begin{proposition}
\label{CommRelations}
The isotopies that yield the commutativity relations in $\pi_0\L^0_2$ (which by Theorem \ref{hypo} are \emph{all} the relations in $\pi_0 \L^0_2$) can be realized as paths in the spaces $\I(c_1,...,c_n;2)$, where $c_i \in \{1,2\}$.
\end{proposition}

\begin{proof}
Note that by Theorem \ref{hypo} any 2-string link can be obtained from infections of the trivial 2-string link by prime knots and non-central prime 2-string links; these infections can be chosen to commute with each other (so that they can be carried out ``all at once'').   In terms of fat string links in $\widehat{\L}_2^0 \left(\simeq \L_2^0\right)$, we can express these operations using a relatively small class of 2-holed mufflers and hockey pucks, as follows.

\begin{figure}
\includegraphics[scale=.55]{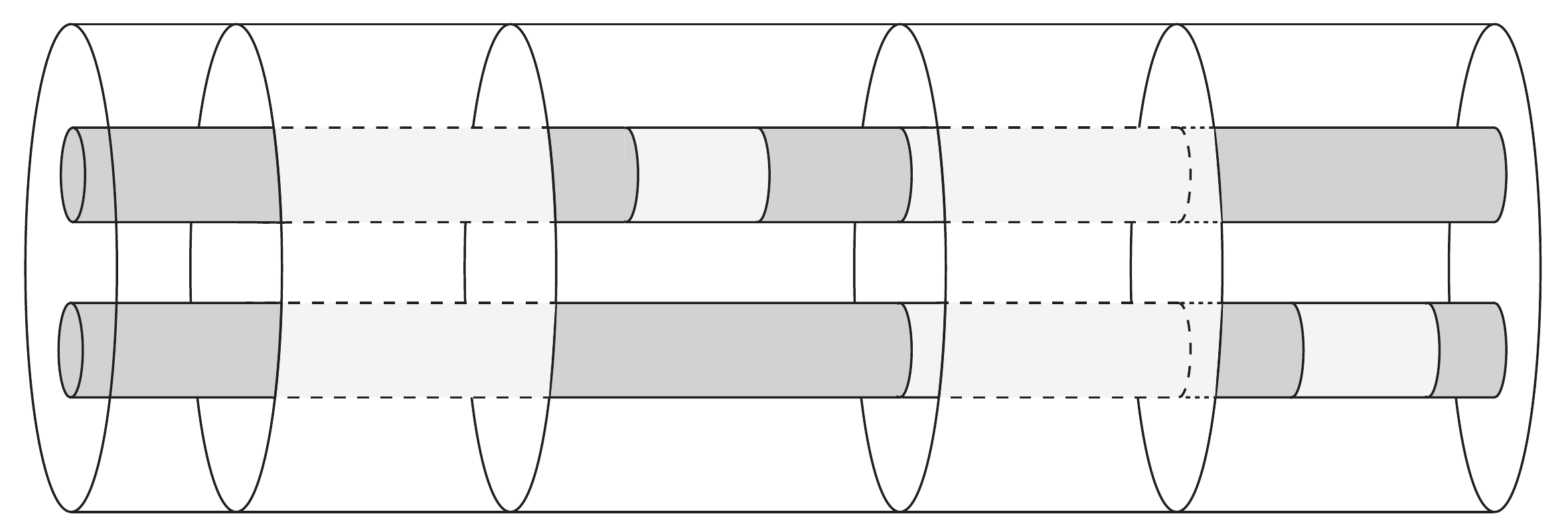}
\caption{An element of $\I(1,1,2,1;2)$, where the $L_i$ are ordered from left to right.  In this element, we have a hockey puck of type (B), then a hockey puck of type (A1), then a two-holed muffler, then a hockey puck of type (A2).}
\label{OperadEncodesCommRelations}
\end{figure}

Recall that an element $a \in \C_1(1)$ is just an affine-linear map $I \incl I$.
Let $e_1, e_2\co D^2 \incl D^2$ denote the restrictions of
the trivial fat 2-string link $i_2\co \coprod_2 I \x D^2 \incl I \x D^2$
to the two components in the 0-time slice:
\[
e_1 \sqcup e_2 \co  (\{0\} \x D^2) \sqcup (\{0\} \x D^2) \incl \{0\} \x D^2.
\]
(Equivalently, $e_1, e_2$ are the restrictions of $i_2$ to the two components of a time-slice at any time $t\in I$).
Consider infection diagrams $(L_0, M_1,...,M_n, \sigma)$ representing classes in $\I(c_1,...,c_n;2)$ which satisfy the following three conditions (see Figure \ref{OperadEncodesCommRelations}):
\begin{itemize}
\item
$L_0$ is the trivial 2-string link.
\item
If $c_i=1$, then either
\begin{itemize}
\item[(A1)]
$L_i = a_i \x e_1$ for some $a_i \in \C_1(1)$ or
\item[(A2)]
$L_i = a_i \x e_2$ for some $a_i \in \C_1(1)$ or
\item[(B)]
$L_i = a_i \x id \co I \x D^2 \incl I \x D^2$ for some $a_i \in \C_1(1)$.
\end{itemize}
\item
If $c_i=2$, then $L_i = a_i \x id \co I \x D^2 \incl I \x D^2$ for some $a_i \in \C_1(1)$.
\end{itemize}

Notice that plugging knots into pucks of types (A1) and (A2) produces a split link, while plugging a knot into a puck of type (B) produces a cable.  Hockey pucks of types (A1) and (A2) can move through the inside of the two-holed mufflers, while the pucks of type (B) can move through the two-holed mufflers on the outside.  These two motions correspond to the centrality of split links and cables, which by Theorem \ref{hypo} are all the commutativity relations in $\pi_0 \L^0_2$.  This proves the proposition.  (Since $\L_2 \cong \L^0_2 \x \Z$, this is fairly close to a statement about all of $\L_2$.)
\end{proof}

\subsection{A suboperad of the 2-colored restriction}
Let $\I_{\{2\}}$ denote the suboperad of $\I$ corresponding to the color $\{2\} \subset \N^+$.  Note that $\I_{\{2\}}$ is an ordinary operad.

\begin{definition}
\label{StackingSuboperad}
We define the \emph{stacking suboperad} $\I_\# \subset \I_{\{2\}}$ as the suboperad
where each space $\I_\#(n)$
consists of elements of $\I(2,...,2;2)$ represented by infection diagrams $(L_0, M_1,...,M_n, \sigma)$  satisfying the following conditions:
\begin{itemize}
\item
$L_0$ is the trivial $2$-string link.
\item
$M_i = a_i \x id \co I \x D^2 \incl I \x D^2$ for some $a_i \in \C_1(1)$.
\end{itemize}
\end{definition}

\begin{figure}
\begin{picture}(216,130)
\put(18,0){\includegraphics[scale=0.55]{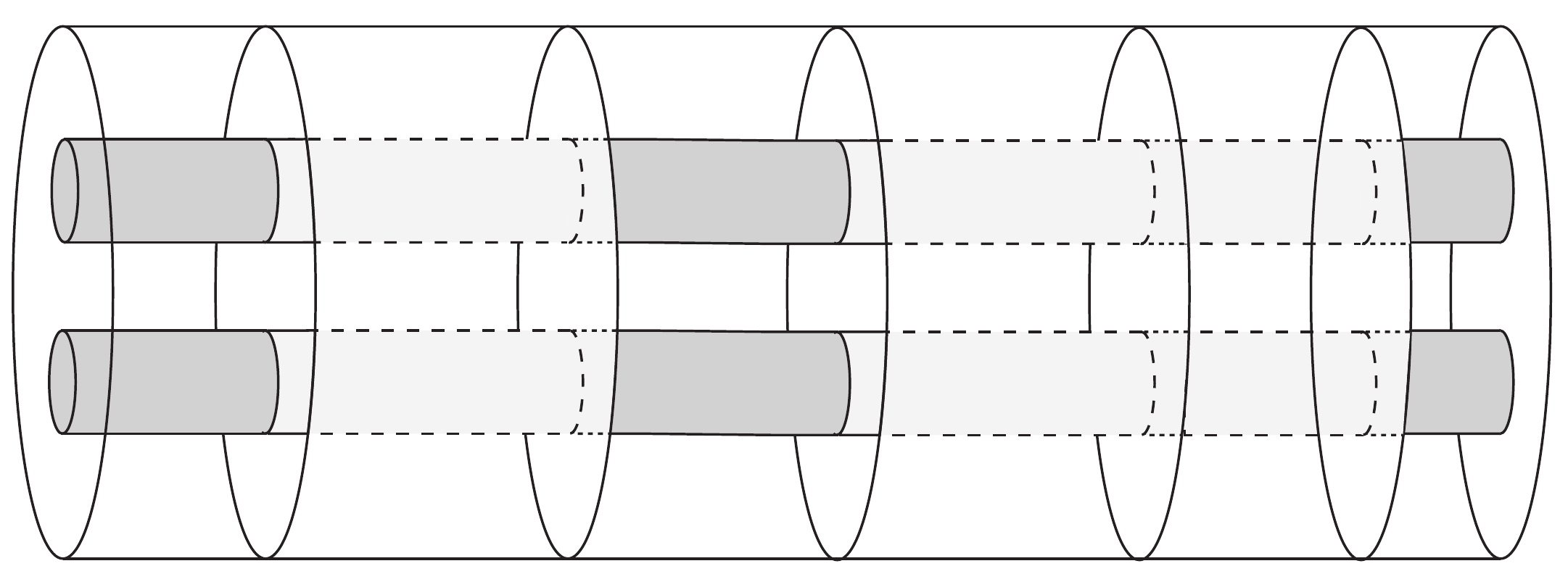}}
\end{picture}
\caption{An element of $\I_\#(3)$ ($\cong \C_1(3)$).}
\label{StackingSuboperad}
\end{figure}

The following is obvious:

\begin{proposition}
\label{Little1Cubes}
For each $n$, the space $\I_\#(n)$ is homeomorphic to $\C_1(n)$.  Thus $\I_\#(n)$ has contractible components and is equivalent to $\Sigma_n$.
\qed
\end{proposition}


\subsection{A decomposition theorem}
Recall that $\LL_c$ is the space of fat $c$-string links with zero framing number in each component; $\LL^0_c \subset \LL_c$ is the subspace where the linking number is 0 (defined in Section \ref{RemovingTwists}); and we have homotopy equivalences $\LL_c \simeq \L_c$ and $\LL^0_c \simeq \L^0_c$.
Let $\P_c \subset \LL_c$ be the subspace of prime $c$-component fat string links.
We will decompose a certain subspace of $\LL_2$ in terms of our infection operad and the prime links in this subspace.

\begin{definition}
Define $\S_2$ to be the subspace of $\LL_2$ consisting of certain components of $\LL_2$: the component of $\LL_2$ corresponding to a string link $L$ is in $\S_2$ if and only if $L$ is a product of prime string links, each of which is \emph{not} in the center of $\pi_0\L_2$. (In other words, each prime factor of $L$ is neither a split link nor a cable.)
Let $\mathcal{PS}_2:= \P_2 \cap \S_2$, let $\S^0_2 := \S_2 \cap \LL^0_2$, and let $\mathcal{PS}^0_2 := \mathcal{PS}_2 \cap \LL^0_2 = \P_2 \cap \S_2 \cap \LL^0_2$.
\end{definition}

Before stating our decomposition theorem, we review a useful lemma, well known to embedding theorists.  Before proving the lemma, we need to set some more definitions.
\begin{itemize}
\item
Let
$\LL=\coprod_{c=1}^\infty \LL_c$.
\item
For $L\in \LL$, let $\LL(L)$ denote the component of $L$ in $\LL$.
\item
Recall that if $L$ is an embedding of a 3-manifold with boundary into $I\x D^2$, $C_L:=D^3 \setminus \overset{\circ}{(\im L)}$, where we identify $I\x D^2$ with $D^3$.
\item
For a manifold with boundary $M$, let $\Diff(M; \d)$ denote the space of diffeomorphisms of $M$ which are the identity on the boundary.
\item
For a group $G$, let $BG$ denote the classifying space of $G$.
\end{itemize}

\begin{lemma}
\label{LinkSpaceCompsAreBGs}
For any $L \in \LL$, $\LL(L) \simeq B\Diff(C_L; \d)$.
\end{lemma}
\begin{proof}
Given a diffeomorphism in $\Diff(D^3, \d)$, we can restrict to the image of $L$ to get a fibration
\[
\xymatrix{\Diff(C_L; \d) \ar[r] & \Diff(D^3; \d) \ar[r] & \LL(L).}
\]
Hatcher showed that $\Diff(D^3; \d)$ is contractible (the Smale Conjecture \cite{HatcherSmaleConj}), which implies the result.
\end{proof}

\begin{theorem}
\label{DecompThm}
The subspace $\S^0_2$ is freely generated over the
stacking suboperad $\I_\#$
by its subspace $ \PS^0_2$ of non-split, non-cable prime string links.
More precisely,
\[
\S^0_2 \simeq
\I_\#(\PS^0_2 \sqcup \{*\}) :=
\coprod_{n=0}^\infty  \I_\#(n) \x_{\Sigma_n} (\PS^0_2 \sqcup \{*\})^n
\left( \simeq \coprod_{n=0}^\infty (\PS^0_2 \sqcup \{*\})^n \right)
\]
where $\{ * \}$ corresponds to the component of the trivial 2-string link.  Furthermore $\S_2 \cong \S^0_2 \x \Z$.
\end{theorem}

\begin{proof}
First note that by Theorem \ref{hypo} we have a bijection on $\pi_0$.  In fact, a prime decomposition $L=L_1\#...\#L_n$ corresponds to an isotopy class of an equivalence class of infection diagram in $\I_\#$ with $n$ mufflers exactly as in Definition \ref{StackingSuboperad}.

Now we  will check that we have an equivalence on each component of $\S^0_2$.  So fix $L \in \S^0_2$. 
Let $C_L$ denote the complement of $L$ in $D^3$, as above.

\begin{definition}
For a $c$-component string link $L$, a \emph{decomposing disk} $D \subset C_L$ is a 2-disk with $c$ open 2-disks removed which is properly embedded in $C_L$ in such a way that $c$ of its boundary components are (isotopic to) the $c$ meridians of $L$.
\end{definition}
Note that a decomposing disk $D$ is incompressible in $C_L$ \cite[Lemma 2.9]{StringLinkMonoid}.  

A prime decomposition $L= L_1\#...\#L_n$ corresponds to a maximal collection of decomposing disks $D_1,...,D_{n-1}$ such that no two $D_i$ are isotopic.  
Thus the decomposing disks $D_1,...,D_{n-1}$ cut $C_L$ into $n$ pieces that are precisely $C_{L_1},...,C_{L_n}$.
Recall the uniqueness of prime decompositions for $L \in \S^0_2$ given by Theorem \ref{hypo}.  The proof of this theorem implies that (the image of) such a maximal collection of decomposing disks is unique up to isotopy.  Note that the prime factors of $L\in \S^0_2$ cannot even be reordered.  

Now consider the fibration
\begin{equation}
\label{DiffDiffEmb}
\xymatrix{
\Diff\left(\coprod_{i=1}^n C_{L_i}; \d \right) \ar[r] &  \Diff(C_L; \d) \ar[r] &
\Emb \left(\coprod_{i=1}^n D_i, \,C_L \right)
}
\end{equation}
Hatcher proved \cite{HatcherIncomprSurfs} that for a 3-manifold $M$ and a properly embedded incompressible surface $S\subset M$,
the space $\Emb(S,M)$ has contractible components unless $S$ is a torus.  (Strictly speaking, Hatcher proves this for connected $S$, but for $S=\sqcup_{i=1}^n S_i$ with each $S_i$ a connected surface, one can use the fibration 
\[
\xymatrix{
\Emb\left(S_n,\, M \setminus \left(\coprod_{i=1}^{n-1} S_i\right)\right) \ar[r] &
\Emb\left(\coprod_{i=1}^n S_i,\, M\right) \ar[r] &
\Emb\left(\coprod_{i=1}^{n-1} S_i,\, M\right)
}
\]
and induction on $n$ to get the result, noting that Hatcher's theorem applies when the 3-manifold is a component of $M \setminus (\sqcup_{i=1}^{n-1} S_i)$.)

Thus the components of the base space in (\ref{DiffDiffEmb}) are contractible.  Since the images of the $D_i$ are determined up to isotopy, we may replace $\Emb(\coprod_{i=1}^n D_i, C_L)$ by $\Diff(\coprod_{i=1}^n D_i)$ (since the latter space also has contractible components).  Note that the fiber in (\ref{DiffDiffEmb}) is $\prod_{i=1}^n \Diff(C_{L_i})$.  So we have 
\[
\xymatrix{
\Diff(\coprod_{i=1}^n D_i) \ar[r] & \Diff(C_L; \d) \ar[r] & \Diff(\coprod_{i=1}^n D_i).
}
\]

Now apply the classifying space functor $B(-)$ to the above fibration.  By Lemma \ref{LinkSpaceCompsAreBGs}, we get 
\[
\xymatrix{
\prod_{i=1}^n \LL(L_i) \ar[r] &  \LL(L) \ar[r] & \prod_{i=1}^n \mathrm{Conf}_2(D^2)
}
\]
where $\mathrm{Conf}_2(D^2)$ is space of ordered distinct pairs in $D^2$ (or the classifying space of the braid group on two strands).  The base space is a $\mathrm{K}(\pi,1)$, i.e., it has trivial $\pi_i$ for $i>1$.  We claim that on $\pi_1$, the fibration is the zero map: in fact, if $\alpha \in \pi_1(\LL(L))$ produced a nontrivial braid (say, in the $i^\mathrm{th}$ factor), then in $\alpha(1)$, at least one of the two summands determined by $D_i$ would have nonzero $\ell$ (number of twists), contradicting that $\alpha$ is a loop (in $\S_2^0$).

So by the long exact sequence in homotopy groups for a fibration, the map from fiber to total space is an isomorphism on $\pi_i$ for all $i\geq 0$.  Then by the Whitehead Theorem,
\[
\LL(L) \simeq \prod_{i=1}^n \LL(L_i).
\]
The right-hand space can be rewritten as $\Sigma_n \x_{\Sigma_n} \prod_{i=1}^n \LL(L_i)$, which by Proposition \ref{Little1Cubes} is equivalent to $ \I_\#(n) \x_{\Sigma_n} \prod_{i=1}^n \LL(L_i)$.  This proves the main assertion of the theorem.

The remaining assertion, that $\S_2 \cong \S^0_2 \x \Z$, follows immediately from Section \ref{RemovingTwists}.
\end{proof}

\subsection{Final remarks and future directions}
We have described the components of links in $\S_2$ in terms of the components of the prime links in $\S_2$.  In general, we do not have descriptions of the components of the prime links in $\S_2$ themselves.
However, we can describe some components of $\L_2$.  We believe that at least some of these descriptions have been known to experts.

\begin{proposition}
The component of a 2-string link $R \in \LL_2$ which is a rational tangle is contractible.
\end{proposition}
\begin{proof}
We have a fibration
\begin{equation}
\label{UnlinkFibn}
\xymatrix{ \Diff(C_R; \d) \ar[r] & \Diff(D^3; \d) \ar[r] & \LL(R)}
\end{equation}
given by restricting to the image of $R$.  The total space is contractible by the Smale Conjecture.
So it suffices to show that the fiber $\Diff(C_R; \d)$ is contractible.  Note that $C_R$ is a genus-2 handlebody.

We claim that for any 3-dimensional handlebody $H$, $\Diff(H; \d)$ is contractible.  This can be proven by induction on the genus.  The basis case of genus 0 is the Smale Conjecture.  For the induction step, let $S$ be a meridional disk in $H$.  Consider the fibration 
\[F\to \Diff(H; \d) \to \Emb(S,H)\]
where the base is the space of proper embeddings of $S$ with fixed behavior on $\d S$.  The fiber $F$ is the space of diffeomorphisms of a handlebody whose genus is 1 less than that of $H$, and it is contractible by the induction hypothesis.  Hatcher's result on incompressible surfaces says that $\Emb(S,H)$ has contractible components.
Furthermore, we claim that any two such embeddings of $S$ are isotopic; this can be proven using the fact that handlebodies are irreducible (i.e., every 2-sphere in $H$ bounds a 3-ball) and standard ``innermost disk'' arguments from 3-manifold theory.  Hence $\Emb(S,H)$ is connected, hence contractible.  
Thus $\Diff(H; \d)$ is contractible.  Thus the base space in the fibration (\ref{UnlinkFibn}) is also contractible.
\end{proof}

Recall the definitions of split links and splitting disks from Definition \ref{SplitAndCable}.

\begin{proposition}
\label{SplitLinkProp}
If $L$ is a split string link which splits as links $L_1, L_2$, then $\LL(L) \simeq \LL(L_1) \x \LL(L_2)$.
\end{proposition}
\begin{proof}
Let $D$ be a splitting disk for $L$.  Consider the fibration
\[
\Diff(C_{L_1}; \d) \x \Diff(C_{L_2}; \d) \to \Diff(C_L; \d) \to \Emb(D, C_L)
\]
where $\Emb(D, C_L)$ is the space of embeddings of $D$ which agree on $\d D$ with the given embedding of $D$.  By Hatcher's theorem on incompressible surfaces, this space has contractible components.  Irreducibility of $C_L$ implies further that any two such embeddings of $D$ in $C_L$ are isotopic, showing that the base space is connected, hence contractible.
This gives us the desired equivalence.
\end{proof}

If we restrict our attention to 2-string links, the split links are just those links which are obtained by tying a knot in one or both strands.  So Budney's work \cite{BudneyTop} together with Proposition \ref{SplitLinkProp} gives a description of the homotopy type of each such component of $\L_2$.

We conclude by mentioning two open problems that immediately stand out as follow-ups to Theorem \ref{DecompThm}:
\begin{itemize}
\item[(1)]
to determine the homotopy types of components of prime non-central 2-string links.
\item[(2)]
to understand how different types of 2-string links interact, i.e., find a generalization of Theorem \ref{DecompThm} from the subspace $\S_2$ to the space of all 2-string links.
\end{itemize}

 \bibliographystyle{plain}
\bibliography{refs}

\end{document}